\documentclass[10pt,leqno]{article}

\columnwidth\textwidth

\usepackage[centertags,sumlimits]{amsmath}
\usepackage{amssymb}
\usepackage{amsthm}

\newtheorem{thm}{Theorem}[section]       \newtheorem{lem}[thm]{Lemma}
\newtheorem{prop}[thm]{Proposition}      \newtheorem{cor}[thm]{Corollary}
\newtheorem{rem}[thm]{Remark}
\newtheorem*{defn}{Definition}           

\newcommand{\ca}{{\cal A}}    \newcommand{\cb}{{\cal B}}   \newcommand{\cd}{{\cal D}}    
\newcommand{\ce}{{\cal E}}

\newcommand{\R}{\mathbb{R}}   \newcommand{\N}{\mathbb{N}}  
 \newcommand{\D}{\mathbb{D}}

\numberwithin{equation}{section}

\begin{document}

\newcommand{\ii}{\^\i}
\newcommand{\ai}{\^a}
\newcommand{\di}{\mbox{\sf div}}
\newcommand{\adj}{\mbox{\sf adj}}
\newcommand{\vre}{{\cal E}}
\newcommand{\pp}{\mbox{\sf supp}}

\noindent 
{\Large\bf Nonlinear Dirichlet forms associated with quasiregular mappings}\\

\noindent
{Camelia Beznea}\footnote
{National University of Science and Technology Politehnica Bucharest,  Splaiul Independentei 313, RO-77206 Bucharest, 
Romania (e-mail: camelia.beznea@upb.ro)},
Lucian Beznea\footnote{Simion Stoilow Institute of Mathematics  of the Romanian Academy,
 Research unit No. 2, 
P.O. Box \mbox{1-764,} RO-014700 Bucharest, Romania, and 
National University of Science and Technology Politehnica Bucharest
(e-mail: lucian.beznea@imar.ro)},
and 
Michael R\"ockner\footnote{Fakult\"at f\"ur Mathematik, Universit\"at Bielefeld,
Postfach 100 131, D-33501 Bielefeld, Germany, and Academy for Mathematics and Systems Science, CAS, Beijing 
(e-mail: roeckner@mathematik.uni-bielefeld.de)}

\vspace{6mm}

\noindent
{\small {\bf Abstract.} 
If $({\cal E}, {\cal D})$ is a symmetric, regular, strongly local
Dirichlet form on $L^2 (X,m)$, admitting a carr\'{e} du champ operator $\Gamma$,
and $p>1$ is a real number, then one can define a nonlinear
form ${\cal E}^p$ by the formula
$$
{\cal E}^p(u,v) = \int_{X} \Gamma(u)^\frac{p-2}{2} \Gamma(u,v)dm ,
$$
where $u$, $v$ belong to an appropriate subspace of the domain  ${\cal D}$.
We show that ${\cal E}^p$ is a nonlinear Dirichlet form in the sense
introduced by P. van Beusekom. 
We then construct the associated Choquet capacity.
As a particular case we obtain the nonlinear form associated with the $p$-Laplace operator 
on $W_0^{1,p}$.
Using the above procedure, 
for each $n$-dimensional quasiregular mapping $f$ we construct a nonlinear Dirichlet 
form $\ce^n$ ($p=n$) such that the components of $f$ become harmonic functions with
respect to $\ce^n$. 
Finally, we obtain  Caccioppoli type inequalities in the intrinsic metric induced by $\ce$,  for harmonic functions with respect to the form  $\ce^p$.
}
\vspace{3mm}

{\small

\vspace{2mm}

}

\section{Introduction}\label{sect1}  

Let  $\Omega\subset\R^{n}$ be a bounded Euclidean domain  and  $p>1$.
If  $u\in C^{2}(\Omega)$ then recall that the {\it $p$-Laplacean}  $\Delta_{p}u$ of $u$  is defined as
$$
\Delta_{p}u :=-\di(\vert\nabla u\vert^{p-2}\nabla u)
$$
and  for all $u,v\in C^{2}_{c}(\Omega)$ we have
$$
\int_{\Omega}v\Delta_{p}udx=
\int_{\Omega}\vert\nabla u
\vert^{p-2}(\nabla u,\nabla v)dx.
$$
This equality allows  to define  a nonlinear bilinear form associated with 
the $p$-Laplace operator:
\begin{equation} \label{1.1}
\langle \Delta_{p}u, v \rangle :=
\int_{\Omega}\vert\nabla u\vert^{p-2}
(\nabla u,\nabla v)\quad \mbox{ for all }\,
u,v\in W_{0}^{1,p}(\Omega).
\end{equation} 

In \cite{Beu}  Petra van Beusekom introduced and studied an abstract class of nonlinear Dirichlet forms on a Banach space,
 in order to  have as models the above form generated by  the $p$-Laplace  operator,  but also the one generated by the
Monge-Amp\`{e}re operator.

Our first aim is to consider a wider class of concrete nonlinear Dirichlet spaces. 
More precisely, if $({\cal E}, {\cal D})$ is a symmetric, regular, strongly local
Dirichlet form on $L^2 (X,m)$, admitting a carr\'{e} du champ operator $\Gamma$,
and $p>1$ is a real number, then one can define a nonlinear
form ${\cal E}^p$ by the formula
$$
{\cal E}^p(u,v) = \int_{X} \Gamma(u)^\frac{p-2}{2} \Gamma(u,v)dm ,
$$
where $u$, $v$ belong to an appropriate subspace of the domain  ${\cal D}$.
We prove that ${\cal E}^p$ is a nonlinear Dirichlet form in the sense
introduced by P. van Beusekom and as a particular case we obtain the nonlinear form associated with the $p$-Laplace operator  on $W_0^{1,p}$. 
We can then apply the strategy  from \cite{HKM} to construct the associated capacity and 
we show that it is a Choquet capacity on $X$, extending the result obtained in  \cite{HKM} for the $p$-Laplace operator.
Following \cite{Beu},  we also construct the equilibrium potential for a compact set.
Finally, we prove Caccioppoli type inequalities in the intrinsic metric induced by $\ce$ for $p$-harmonic functions, that is, harmonic with respect to the form  $\ce^p$.

In the last four decades results from  nonlinear potential theory have been used in the study of quasiconformal and quasiregular mappings; 
cf. \cite{BI}, \cite{IM},  and \cite{HKM}.
Our second aim is to use  the  obtained results for each $n$-dimensional quasiregular mapping $f$ and to construct a nonlinear Dirichlet 
form $\ce^n$ ($p=n$) such that the components of $f$ become harmonic functions with respect to $\ce^n$.
This statement should be compared  with the results from the monograph \cite{HKM}, 
where to a quasiregular mapping it is associated a different structure, namely a nonlinear harmonic space.
We apply the obtained Caccioppoli type inequalities to the quasiregular mapping and we discuss the connections with the results from \cite{BI}.

The structure of the paper is the following. 
In Section \ref{sect2} we present, following \cite{Beu}, the nonlinear Dirichlet forms, and as an  example the form $(\ref{1.1})$  associated with the $p$-Laplace operator.
In Section \ref{sect3} we study the nonlinear $p$-form  generated by a linear regular Dirichlet form. 
The main result is Theorem \ref{thm3.8}, showing that this form is a nonlinear Dirichlet form in the sense of Section \ref{sect2}.
As a consequence, we prove in Corollary \ref{cor3.9} that the $p$-form associated with the Dirichlet form given by a uniformly  
elliptic differential operator is a nonlinear Dirichlet form on 
$ W_0^{1,p} (\Omega) $, in particular, the Dirichlet form associated with the $p$-Laplace operator. 
We construct then the equilibrium potential and the induced capacity. 
The nonlinear Dirichlet $p$-form  associated with a quasiregular mapping is investigated in Section \ref{sect4}. 
The main result is  Theorem \ref{thm4.3},
 proving   that  the components of a quasiregular mapping are harmonic functions with respect to a nonlinear Dirichlet $\ce^n$, where  $\ce$ is an associated linear Dirichlet form. 
 The Caccioppoli type inequalities and the application to the quasiregular mappings are exposed in Section \ref{sect6}. 
 We collected in an Appendix basic results used in the paper on local and strongly local Dirichlet forms: the capacity and quasi-continuity, the energy measure, and the carr\'e du champ operator.
 We also put in the Appendix basic  facts on  quasiregular mappings.
 
Finely, we would like to mention that concerning our results in Section \ref{sect3} there is a recent related paper  
\cite{K} that we learned of recently at a conference in Chemnitz, 
where the second named author presented a talk  about  the results of our paper.  
We  would like to thank Kazuhiro  Kuwae for  sending us a preliminary version of his paper, before it appeared on arXiv.
We refer to assertion $4)$ of Remark \ref{rem3.13} below,
where we explain the exact relation of \cite{K} with the results in Section \ref{sect3}  of our paper.

\section{Nonlinear Dirichlet forms} \label{sect2}

We present in this section the basic facts about the
nonlinear Dirichlet forms, in the form
developed by P. van Beusekom in \cite{Beu}, in order to
include as models the forms generated by  the 
Monge-Amp\`{e}re and $p$-Laplace  operators.

We fix a system  $(L^{1}(X,m),B,\mbox{\bf K},A)$, where:

\noindent
$-$ $X$  is a locally compact separable metric space
and $m$ is a Radon   measure on $X$. 

\noindent
$-$ $B$ is a real Banach space, with the norm  $\|\cdot \|_{B}$,
continuously embedded in  $L^{1}(X,m)$.
We denote by  $\langle, \rangle$
the duality between the dual  $B'$ of $B$ and $B$. 

\noindent
$-$ $\mbox{\bf K}\subset B$ 
is a closed convex cone such that: \\
$1.$ \quad $0\in \mbox{\bf K}$.\\
$2.$ \quad $u,v\in \mbox{\bf K}\Longrightarrow u\land v\in \mbox{\bf K}$. \\
$3.$ \quad $u,v\in \mbox{\bf K}$, $\alpha\in\R_{+}\Longrightarrow u
\land (v+\alpha)\in \mbox{\bf K}$. \\

\noindent
$-$ $A: \mbox{\bf K}\longrightarrow B'$ 
is an operator having the following properties: \\
$1.$ \quad $A0=0$. \\
$2.$ \quad $A$ is{ \bf coercive}, i.e. 
$\langle Au,u \rangle \geq c \| u \|_{B}^{p}$
(where $c$ and  $p$ are strictly positive constants)
for all  $u\in \mbox{\bf K}$. \\
$3.$ \quad $A$ is {\bf strictly monotone}, that is 
$\langle Au-Av, u-v \rangle > 0$,
for all  $u,v\in \mbox{\bf K}$, $u\not=v$.

\noindent
$-$ The triple $(A, \mbox{\bf K}, B)$ satisfies the {\bf Browder property}
for each $f\in B'$:
if $W$ is a nonempty convex closed subset of $\mbox{\bf K}$ 
then there exists  $u\in W$ such that
\begin{equation} \label{2}
\langle Au-f,v-u \rangle \geq 0\quad
\mbox{ for all } \, v\in W.
\end{equation} 

The map
$$
(u,v)\longmapsto \langle Au,v \rangle ,\quad
u\in \mbox{\bf K},\quad v\in B,
$$
associated to an operator $A$ as before, is called 
{\bf monotone  form} (on $B$).

An element $u\in \mbox{\bf K}$ is called {\bf pure  potential } if
$$
\langle Au,v \rangle \geq 0\quad
\mbox{ for all }\, v\in \mbox{\bf K},\quad v\geq 0.
$$

A monotone form is called {\bf nonlinear Dirichlet form}
if for each two potentials $u$ \c and $v$
and every constant $\alpha>0$, the following condition hold:
$$
\leqno{D1}\quad\quad \langle A(u\land v),u-u\land v \rangle \geq 0 .
$$
$$
\leqno{D2}\quad\quad \langle A(u\land (v+\alpha )),
u-u\land (v+\alpha) \rangle \geq 0.
$$

A normal contraction  $T$ {\bf operates on }
$A$ if
$$
u,v\in \mbox{\bf K}\Longrightarrow T(u)\in 
\mbox{\bf K}\mbox{ and } \langle A(u+Tu+v)-Av, u-Tu\rangle \geq 0.
$$

A normal contraction  $T$  $C$-{\bf operates  on} $A$
if
$$
u,v\in \mbox{\bf K}\Longrightarrow T(u)\in \mbox{\bf K}\mbox{ and }
\langle A(v+Tu)-Av, u-Tu \rangle \geq 0.
$$

\noindent
{\bf The nonlinear Dirichlet form associated with  the
$p$-Laplace operator.} 
Let  $\Omega\subset\R^{n}$ be a bounded domain  and  $p>1$.
If  $u\in C^{2}(\Omega)$ then the {\bf $p$-Laplacean}  $\Delta_{p}u$ of $u$  is defined by 
$$
\Delta_{p}u=-\di(\vert\nabla u\vert^{p-2}\nabla u)
$$
and  for all $u,v\in C^{2}_{c}(\Omega)$ we have
$$
\int_{\Omega}v\Delta_{p}udx=
\int_{\Omega}\vert\nabla u
\vert^{p-2}(\nabla u,\nabla v)dx.
$$
This equality allows  the following definition of the nonlinear form associated with the 
the $p$-Laplace operator:
$$
\Delta_{p}u, v \rangle :=
\int_{\Omega}\vert\nabla u\vert^{p-2}
(\nabla u,\nabla v)\quad \mbox{ for all }
\,
u,v\in W_{0}^{1,p}(\Omega).
$$
In this case we have
$$
\mbox{\bf K}=B=W_{0}^{1,p}(\Omega) \mbox{ and }\quad A=\Delta_{p}.
$$
The main result from \cite{Beu}, Ch. 5, is the following:
{\it the nonlinear form associated with the 
$p$-Laplace   operator is a Dirichlet  form on  $W_{0}^{1,p}(\Omega)$.}
(See Theorem 5.2.7 in \cite{Beu}.)
We show in the sequel that this example is a particular case of a wider
class of nonlinear Dirichlet forms.

\section{\bf The nonlinear $p$-form  generated by a linear regular Dirichlet form}\label{sect3} 

In this section we construct, starting with a
strongly local regular Dirichlet form  $(\ce,\cd(\ce))$, admitting a 
carr\'{e} du champ operator $\Gamma$,  a nonlinear Dirichlet form, for each real number  $p>2$.
We suppose  that $m$   is a finite measure on $X$ and 
let $\Gamma$ be the associated  carr\'{e} du champ operator.
If $p>2$ is  a fixed real number then we  denote by
$\|\quad\|_{p}$  the norm from $L^{p}(X,m)$,
$\| u\|_{p}:= ( \int_{X}\vert u\vert^{p}dm )^{\frac{1}{p}}.$

Following \cite{BH1} and \cite{BH2} we define: 
$$
\cd_{p}:= \{ u\in\cd(\ce)\cap L^{p}(X,m) :\; 
\Gamma (u)^{\frac{1}{2}}\in L^{p} (X,m) \} .
$$
For each $\cb$-measurable function  $u$ we consider
$$
\| u\|_{\cd_{p}}:= 
( \int_{X}\vert u\vert^{p}dm +
\int_{X}\Gamma (u)^{\frac{p}{2}} dm )^{\frac{1}{p}}
= \left(\| u\|_{p}^{p} + \|\Gamma (u)^{\frac{1}{2}}\|_{p}^{p}
\right)^{\frac{1}{p}}.
$$

\noindent
\begin{prop}  \label{prop3.1} 
The following assertions hold.

$a)$
$(\cd,\|\quad\|_{\cd_{p}})$ is a normed real vector space.

$b)$
$\cd_{p}$ is a vectorial lattice:
if $u,v\in\cd_{p}$ then $ u\land v,$ $u\lor v\in\cd_{p}$.

$c)$
If  $u\in\cd_{p}$ then  $u^{+}\land 1\in\cd_{p}$. 

$d)$
If $u,v\in\cd_{p}\cap L^{\infty}(X,m)$,  then $u\cdot v\in\cd_{p}$. 

$e)$
If $u\in\cd(\ce)$ \c and  $\Gamma (u)\in L^{\infty}(X,m)$
then $u\in\cd_{p}$. 

$f)$
Let  $u\in\cd_{p}$  and $u_{n}:=\left((-n)\lor u\right)\land n$, 
$n\in\N$.
Then $(u_{n})_{n}\subset\cd_{p}$
  and $u_n\longrightarrow u$ in $\cd_{p}$. 
\end{prop}

\begin{proof}
$a)$ 
Let $u,v\in\cd_{p}$.
Then $\Gamma (u)^{\frac{1}{2}}$, $\Gamma (v)^{\frac{1}{2}}\in L^{p}
(X,m)$ and from  $(\ref{6.4})$ we deduce that 
$\Gamma(u+v)^{\frac{1}{2}}\in L^{p}(X,m)$ and therefore
$u+v\in\cd_{p}$.
Consequently  $\cd_{p}$ is a vector space.
Since  $\| u\|_{p}\leq\| u \|_{\cd_{p}}$ we get that if 
 $  \| u \|_{\cd_{p}}=0$ then $ u=0$.
It remains to show the triangle inequality for 
$\| \;\, \|_{\cd_{p}}$. 
Again from  $(\ref{6.4})$  and then using
Minkovski inequality we get :
$$
\| u+v \|_{\cd_{p}}^{p}= 
\| u+v \|_{p}^{p} + \|\Gamma (u+v)^{\frac{1}{2}}\|_{p}^{p}
\leq
\int_{X}[\vert u+v\vert^{p}+
(\Gamma (u)^{\frac{1}{2}}+\Gamma (v)
^{\frac{1}{2}})^{p}]dm\leq
$$
$$
\int_{X}[\left(\vert u\vert^{p}+
\Gamma (u)^{\frac{p}{2}}\right)^{\frac{1}{p}}+
\left(\vert v\vert^{p}+\Gamma (v)^{\frac{p}{2}}\right)^{\frac{1}{p}}]^{p}dm
\leq
\left(\int_{X}(\vert u\vert^{p}+
\Gamma(u)^{\frac{p}{2}})dm\right)^{\frac{1}{p}}+
\left(\int_{X}(\vert v\vert^{p}+
\Gamma(v)^{\frac{p}{2}})dm)^ {\frac{1}{p}}\right)^{p}=
$$
$$
=
\left(\| u\|_{\cd_{p}}+\| v\|_{\cd_{p}}\right)^{p}.
$$

$b)$
From the truncation formula  $(\ref{6.1})$ we  deduce that 
$\Gamma(u\land v)\leq\Gamma(u)+\Gamma(v)$ whenever $u,v\in\cd(\ce)$
and consequently 
$u\wedge v\in\cd_{p}$.

$c)$
Let us remark firstly that if 
$u\in\cd(\ce)$ then, 
since the unit contraction 
operates on $(\ce,\cd(\ce))$, 
we obtain $u^{+}\land 1\in\cd(\ce)$ and we have
\begin{equation} \label{3.1}
\quad\quad \mu_{\langle u^+\land 1\rangle}=1_{[0<\tilde{u}<1]}
\cdot\mu_{\langle u\rangle }\quad
\mbox{ for all } u\in\cd(\ce).
\end{equation}
Indeed, applying $(\ref{6.1})$ we  deduce that 
$$
\mu_{\langle u^+,v \rangle }=1_{[\tilde{u}>0]}\cdot\mu_{\langle u,v \rangle } \mbox{ and }\  
\mu_{\langle u^+ \rangle }=1_[\tilde{u}>0]\cdot\mu_{\langle  u \rangle}\quad
\mbox{ for all } u, v\in\cd(\ce)_{loc}; 
$$
see \cite{St1}.
Again from  $(\ref{6.1})$ and since $\ce$ is strongly local 
($\mu_{\langle 1 \rangle }=0$),
we get
$\mu_{\langle u^+\land 1,u^+\land 1\rangle }=1_{[{\tilde{u}^+}<1]}\cdot\mu_{ \langle u^+ \rangle }
+1_{[\tilde{u}\geq 1]}\mu_{\langle 1 \rangle }=1_{[{\tilde{u}^+}<1]}\cdot\mu_{\langle u^+ \rangle }=
1_{[\tilde{u}>0]}\cdot 1_{[{\tilde{u}^+}< 1]}\cdot\mu_{\langle u \rangle }=
1_{[0<\tilde{u}<1]}\cdot\mu_{\langle u \rangle }
$
and consequently $(\ref{3.1})$ holds.

From $(\ref{3.1})$ we  deduce further that
$\Gamma (u^+\land 1)\leq\Gamma(u)\quad
\mbox{ for all }  u\in\cd(\ce)
$
and therefore
$u^+\land 1\in\cd_{p}$ whenever $u\in\cd_{p}$.  

$d)$
Let now $u,v\in\cd_{p}\cap L^{\infty}(X,m)$.
Since $u\cdot v\in\cd(\ce)$, then from Leibnitz rule we get
$$
\Gamma(u\cdot v)=\Gamma(u\cdot v,u\cdot v)=u\Gamma(v,uv)+v\Gamma(u,uv)=
u^2\Gamma(v)+ v^2\Gamma(u)+2u\cdot v\Gamma(u,v).
$$
As a consequence, since the measure $m$ is finite and $\Gamma(u)$, $\Gamma(v)$, $\Gamma(u,v)\in L^{p}(X,m)$, 
we conclude that
$\Gamma(u\cdot v)^{\frac{1}{2}}\in L^{p}(X,m)$.

$e)$
The assertion follows directly from  $L^{\infty}(X,m)\subset L^{p}(X,m)$.

$f)$
From the above assertion $b)$ we know that $(u_n)_n\subset\cd_{p}$,
$u_n\longrightarrow u$ in $L^{p}(X,m)$ and from the properties of the linear
Dirichlet forms  we also know that $u_n\longrightarrow u$ in $\cd(\ce)$. 
Therefore 
$\lim_{n\rightarrow\infty}\int_{X}\Gamma(u_n-u)dm=0$.
The equality $(\ref{6.2})$ implies
$
\Gamma(u_n)=1_{[-n< \tilde{u}< n]}\cdot\Gamma(u)\leq\Gamma(u)
\ \mbox{ for all }  n
$
and so $\lim_{n\rightarrow\infty}\|\Gamma(u_n-u)^{\frac{1}{2}}\|_{p}=0$.  
\end{proof}

\begin{prop}  \label{prop3.2} 
The following assertions hold.

$1)$ The pair ($\cd_{p}\, ,\|\cdot \|_{\cd_p})$  is a real Banach space.

$2)$ The space $\cd_{p}$ is reflexive.
\end{prop}

\begin{proof}
Assertion $1)$, the fact that $\cd_p$ is a  Banach space,  is stated without proof in \cite{BH1} and \cite{BH2}, Exercise 6.3.
For the reader convenience we presented its proof in Appendix {\bf A2}.

$2)$  By assertion $5)$ of Proposition \ref{prop6.1} it follows that the Dirichlet form $(\ce, \cd(\ce))$ admits a gradient $D$.
Clearly, if $u\in \cd_p$ then $Du$ belongs to $L^p(m; H)$, so, we may 
consider the map $J: \cd_p \longrightarrow \cd(\ce)\times L^p(m; H) $ defined as 
$J(u):= (u, Du)$ for all $u\in \cd_p$.
Let $V:= \cd(\ce)\times L^p(m; H) $, equipped with the norm $\|\cdot \|_V$ of product space,
$\| (u,f)\|_V:= \ce(u)^{\frac 1 2}+ \|  f\|_{L^p(m;H)}$. 
Then for some $c\geq 1$
$$
c^{-1} \| u \|_{\cd_p}\leq \| J(u) \|_V\leq  c\| u\|_{\cd_p}. 
$$
Hence $J: \cd_p \longrightarrow J(\cd_p)$ is a homeomorphism and since by assertion $1)$ the space
 ($\cd_{p}\,  ,\|\cdot \|_{\cd_p})$ is complete, it follows that $J(\cd_p)$ is a closed subspace of $(V, \|\cdot \|_V)$ which is reflexive.
 Therefore  $J(\cd_p)$ is reflexive, and hence  so is $\cd_p$. 
\end{proof}

We can introduce now the nonlinear form associated with a linear 
Dirichlet form.

\begin{defn}
For  $u,v\in\cd_p$  we define
\begin{equation} \label{3.2}
\ce^p(u,v) :=
\int_X\Gamma(u)^\frac{p-2}{2}
\cdot\Gamma (u,v)dm .
\end{equation}
\end{defn}
By $(\ref{6.3})$ we get
$$
\vert\Gamma(u)^\frac{p-2}{2}\cdot\Gamma(u,v)\vert\leq
\Gamma(u)^{\frac{p-1}{2}}
\cdot\Gamma(v)^{\frac{1}{2}}
$$
but $\Gamma(u)^\frac{1}{2}$, $\Gamma(v)^\frac{1}{2}\in L^p(X,M)$ 
and therefore
$\Gamma(u)^\frac{p-1}{2}\in L^q(X,m)$, 
with $\frac{1}{p}+\frac{1}{q}=1$.
Consequently
$\Gamma(u)^\frac{p-2}{2}\Gamma(u,v)\in L^1(X,m)$
that is
$\ce^p(u,v)$  is well defined and it is a real number.
In addition
$$
\ce^p(u,u) = \| \Gamma(u)^\frac{1}{2}\|_{p}^{p}.
$$
The map
$$
\ce^p:\cd_p\times\cd_p\longrightarrow\R
$$
given by  $(\ref{3.2})$, is called the {\bf nonlinear $p$-form} 
associated with the linear Dirichlet form ($\ce,\cd(\ce)$). 

Since the Dirichlet form ($\ce,\cd(\ce)$) admits a gradient (cf. assertion $5)$ of Proposition \ref{prop6.1}) one can write the nonlinear $p$-form $\ce^p$ using the gradient $D$: 
\begin{equation} \label{p-gradient}
\ce^p(u, v)= \int_X \| Du \|^{p-2}_H (Du, Dv)_H \, \mbox{ for all } u, v\in \cd_p.
\end{equation}

\begin{prop}  \label{prop3.3} 
The following assertions hold.

$a)$
The nonlinear $p$-form  $\ce^p$  is {\bf  homogeneous of degree} $p-1$,
that is
$$
\ce^p(tu,v)=t^{p-1}\ce^p(u,v)\quad
\mbox{ for all } \, u,v \in \cd_p , t\in\R_+.
$$

$b)$
The nonlinear $p$-form $\ce^p$ satisfies the  {\bf sector condition}:
$$
\vert\ce^p(u,v)\vert\leq\ce^p(u,u)^\frac{p-1}{p}\cdot\ce^p(v,v)^\frac{1}{p}
\quad \mbox{ for all } \, u,v\in\cd_p.
$$
\end{prop}

\begin{proof}
Assertion $a)$ follows immediately from the bilinearity of the carr\'{e} du champ operator $\Gamma$. 

$b)$ 
Let  $u,v\in\cd_p$.
By $(\ref{3.2})$ and then  using the H\"{o}lder inequality we get
$$
(\ce^p(u,v))^2=
\left(\int_X\Gamma(u)^\frac{p-2}{2}\Gamma(u,v)dm\right)^2 
\leq\left(\int_X\Gamma(u)^\frac{p}{2}dm\right)\cdot
\left(\int_X\Gamma(u)^\frac{p-2}{2}\Gamma(v)dm\right)\leq 
$$
$$
\left(\int_X
\Gamma(u)^\frac{p}{2}dm\right)\cdot\left(
\int_X\Gamma(u)^
\frac{p}{2}dm\right)^\frac{p-2}{p}\cdot\left(
\int_X\Gamma(v)^\frac{p}{2}dm
\right)^\frac{2}{p}= 
$$
$$
\left(
\int_X
\Gamma(u)^\frac{p}{2}dm\right)^{2\frac{p-1}{p}}\cdot\left(
\int_X
\Gamma(v)^\frac{p}{2}dm\right)^\frac{2}{p}=
\left(\ce^p(u,u)^\frac{p-1}{p}\cdot\ce^p(v,v)^\frac{1}{p}\right)^2.
$$
\end{proof}

\noindent
{\bf The definition of the operator } $A=\mbox{\sf L}_p$ \\

Let  $u\in\cd_p$.
The functional
$\mbox{\sf L}_p u:\cd_p\longrightarrow\R$ defined by
$$
\mbox{\sf L}_pu(v):=\ce^p(u,v)
$$
is linear.
From the sector condition (Proposition  3.3 $b)$ ) we deduce that:
$$
\mbox{\sf L}_pu(v)\leq\ce^p(u,u)^\frac{p-1}{p}\cdot\| v\|_{\cd_p}\quad
\mbox{ for all } \, v\in\cd_p .
$$
We conclude that $\mbox{\sf L}_pu$ is an element of the topological dual  
$\cd^{'}_p$ of $\cd_p$.
We have defined in this way an operator 
$$
\mbox{\sf L}_p:\cd_p\longrightarrow\cd_p^{'}
$$
such that
\begin{equation} \label{3.3}
\ce^p(u,v)= \langle \mbox{\sf L}_pu,v \rangle \quad
\mbox{ for all } \, u,v\in\cd_p.
\end{equation}

\begin{prop} \label{prop3.4}  
The operator  $\mbox{\rm \sf L}_p$ is homogeneous of degree  $p-1$ i.e.
$$
\mbox{\rm\sf L}_p(tu)=t^{p-1}\mbox{\rm\sf L}_pu \quad
\mbox{ for all } \, u\in\cd_p,\quad t\in\R_+
$$
and it satisfies the sector condition
\begin{equation} \label{3.4}
\langle \mbox{\rm\sf L}_pu,v \rangle
\leq \langle \mbox{\rm\sf L}_pu,u \rangle^{\frac{p-1}{p}}\cdot 
\langle \mbox{\rm\sf L}_pv,v\rangle^\frac{1}{p}
\quad \mbox{ for all }  \, u,v\in\cd_p .
\end{equation}
\end{prop}

\begin{proof}
The inequality $(\ref{3.4})$ is  precisely the sector condition for the form
$\ce^p$ (Proposition \ref{prop3.3} $b)$), 
expressed using the operator $\mbox{\sf L}_p$. 
Assertion $a)$ from Proposition \ref{prop3.3}  shows that
$\langle \mbox{\sf L}_p(tu),v \rangle =\ce^p(tu,v)=t^{p-1}\ce^p(u,v)=
t^{p-1} \langle \mbox{\sf L}_pu,v \rangle $
and thus $\mbox{\sf L}_p(tu)= t^{p-1}\mbox{\sf L}_pu$.
\end{proof}

\noindent
{\bf The nonlinear Dirichlet form associated with the $p$-Laplace operator }

This classical Dirichlet form is strongly local and admits the carr\'e du champ operator
$$
(u,v)\longmapsto 2\nabla u\cdot\nabla v,\quad u,v\in W_0^{1,2}(\Omega).
$$
 
If $p>1$ is a real number then for $u\in C^2(\Omega)$ we define the $p$-Laplacean $\Delta_pu$ of $u$ by
$$
\Delta_pu=-div (|\nabla u|^{p-2}\nabla u)
$$
and for all $u,v\in C_c^2(\Omega)$ we have
$$
\int_\Omega v\Delta_pudx=\int_\Omega |\nabla u|^{p-2}(\nabla u,\nabla v)dx.
$$
This equality permits to define the nonlinear 
form associated with the $p$-Laplace operator,
\begin{equation} \label{3.5}
\langle  \Delta_pu,v \rangle :=\int_\Omega |\nabla u|^{p-2}(\nabla u,\nabla v)
\quad \mbox{ for all } 
u,v\in W_0^{1,p}(\Omega).
\end{equation}
In this case we have
$$
K=B=W_0^{1,p}(\Omega)\quad\mbox{ and } A=\Delta_p.
$$

The main result from \cite{Beu}, Ch. 5, is the following:
The nonlinear form associated with the $p$-Laplace operator is a Dirichlet
form on $W_0^{1,p}(\Omega)$; see Theorem 5.2.7 in \cite{Beu}.

Recall that the generator of the classical Dirichlet form 
$(\D , W_0^{1,2}(\Omega))$  is the Laplace  operator  $\Delta$,
$$
\D(u,v)=-(\Delta u,v)_2
$$
and the nonlinear  $p$-form  $\D^p$ associated with the form  $\D$
is generated by the 
$p$-Laplace operator $\Delta_p$:
$$
\D^p(u,v)=  \langle \Delta_p u, v \rangle.
$$

Let now $\mbox{\sf L}$ be the generator of the linear  Dirichlet form
$(\ce,\cd(\ce))$, more precisely,
$$
\ce(u,v)=-(\mbox{\sf L}u, v)_2 \quad
\mbox{ for all } \; u\in\cd(L),\; v\in\cd(\ce).
$$

The notation $\Delta_p$ for the  $p$-Laplace  operator
justifies the  notation  $\mbox{\sf L}_p$ 
for the operator generating the nonlinear
$p$-form  $\ce^p$, by the equality $(\ref{3.3})$.\\

We prove further that the  operator $\mbox{\sf L}_p$ 
has  properties that are similar to those of the $p$-Laplace operator,
extending the results from \cite{Beu}, Ch.5.

\begin{thm} \label{thm3.5}. 
The following assertions hold. 

$a)$
The operator $\mbox{\rm\sf L}_p:\cd_p\longrightarrow\cd_p^{'}$ 
is monotone,
$$
\langle \mbox{\rm\sf L}_pu-\mbox{\rm\sf L}_pv, u-v \rangle \geq 0\quad
\mbox{ for all } \, u,v\in\cd_p.
$$
In addition
$$
\langle \mbox{\rm\sf L}_pu-\mbox{\rm\sf L}_pv,u-v \rangle
=0\Longrightarrow\Gamma (u-v)=0
$$
and
\begin{equation} \label{3.6}
\Gamma(u)^\frac{p-2}{2}  \Gamma(u,u-v)- 
\Gamma(v)^\frac{p-2}{2} \Gamma(v,u-v)\geq 0.
\end{equation}

$b)$
If the Dirichlet  form  $(\ce,\cd(\ce))$ is coercive
i.e. there is a constant $k$ such that
$$
\| u\|_2^2\leq k\cdot\ce(u,u)\quad
\mbox{ for all } \, u\in\cd(\ce)  ,
$$
then the operator  $\mbox{\rm\sf L}_p$ is coercive, that is 
$$
\| u\|_{\cd_p}^p\leq c \cdot \langle \mbox{\rm\sf L}_pu, u \rangle \quad
\mbox{ for all } \, u\in \cd_p
$$
where  $c>0$ is a constant.
Particularly,  in this case  $\mbox{\rm\sf L}_p$ is strictly monotone.
\end{thm}

\begin{proof}
We show firstly that inequality $(\ref{3.6})$ holds.\\
Let $u,v\in\cd_p$.
From $(\ref{6.3})$ we deduce that 
$-\Gamma(u,v)\geq -\Gamma(u)^\frac{1}{2} \Gamma(v)^\frac{1}{2}$
and we get
$$
\Gamma(u)^\frac{p-2}{2} \Gamma(u,u-v)- \Gamma(v)^\frac{p-2}{2} \Gamma(v,u-v)=
\Gamma(u)^\frac{p}{2} + \Gamma(v)^\frac{p}{2}-
[\Gamma(u)^\frac{p-2}{2} + \Gamma(v)^\frac{p-2}{2}]\Gamma(u,v)\geq
$$
$$
\Gamma(u)^\frac{p}{2} + \Gamma(v)^\frac{p}{2} -
\Gamma(u)^\frac{p-1}{2} \Gamma(v)^\frac{1}{2} - 
\Gamma(v)^\frac{p-1}{2} \Gamma(u)^\frac{1}{2} =
[\Gamma(u)^\frac{p-1}{2} - \Gamma(v)^\frac{p-1}{2}] \cdot
 [\Gamma(u)^\frac{1}{2}   - \Gamma(v)^\frac{1}{2}  ].
$$
Since the function $t\longrightarrow t^{p-1}$ is strictly monotone on 
$(0,\infty)$ for
$p>1$, we obtain that  $(t^{p-1}-s^{p-1})(t-s)>0$ for
$s,t\in (0,\infty)$, $s\not= t$.
As a consequence we deduce $(\ref{3.6})$.

Integrating  $(\ref{3.6})$ with the measure $m$ we have
$\langle \mbox{\rm\sf L}_p u - \mbox{\rm\sf L}_p v, u-v \rangle \geq 0$, 
i.e.
$\mbox{\sf\rm L}_p$ is monotone.

If $ \langle \mbox{\rm\sf L}_p u - \mbox{\rm\sf L}_p v, u-v \rangle =0$ then 
from the above inequalities we conclude that 
$\Gamma(u)=\Gamma(v)$
and thus
$$
\int_X\Gamma(u)^\frac{p-2}{2}\Gamma(u-v)dm=0.
$$
Therefore $\Gamma(u)\cdot\Gamma(u-v)=0$.
From $(\ref{6.5})$ it results
$\Gamma(u-v)^\frac{1}{2}\leq\Gamma(u)^\frac{1}{2}+\Gamma(v)^\frac{1}{2}
=2\Gamma(u)^\frac{1}{2}
$
and finally
$\Gamma(u-v)=0$.

$b)$
Let us suppose now that the form  
$(\ce,\cd(\ce))$ is coercive
i.e.
$$
\int_X v^2dm\leq k
\int_X\Gamma(v)dm \quad 
\mbox{ for all } \, v\in\cd(\ce).
$$
Let  $u\in\cd_p\cap\mbox{\rm\sf L}^\infty$.
Proposition 6.2.2 in
\cite{BH2} implies $\vert u\vert^\frac{p}{2}\in \cd(\ce)$.
Applying the above inequality for $v=\vert u\vert^\frac{p}{2}$,
the chain rule and
the H\"{o}lder inequality, we get:
$$
\| u\|_p^p=
\int_X\vert u\vert^p dm=
\int_X(\vert u\vert^\frac{p}{2})^2 dm\leq k
\int_X\Gamma(\vert u\vert^\frac{p}{2})dm=
$$
$$
k(\frac{p}{2})^2
\int_X\vert u \vert^{p-2}\Gamma(\vert u\vert )dm
\leq k(\frac{p}{2})^2
(\int_X\vert u\vert^p dm )^\frac{p-2}{p}
( \int_X\Gamma(u)^\frac{p}{2} dm )^\frac{2}{p}=
k(\frac{p}{2})^2\| u\|_p^{p-2}\cdot\|\Gamma(u)^\frac{1}{2}\|_p^2.
$$
Therefore
$$
\| u\|_p^p\leq ( \frac{\sqrt{k}p}{2} )^p \langle \mbox{\sf L}_pu,u \rangle
$$
and thus
$$
\| u\|_{\cd_p}^p\leq (1+(\frac{\sqrt{k} p}{2})^p)
\langle \mbox{\sf L}_pu, u \rangle .
$$
If $u\in\cd_p$, then exists a sequence $(u_n)_n\subset\cd_p\cap L^\infty(X,m)$
such that $u_n\rightarrow u$ \ii n $\cd_p$ 
(c.f.  Proposition  \ref{prop3.1} $f)$).
Applying the last inequality for any function $u_n$ and
then passing to the limit, we conclude that  $\mbox{\sf L}_p$ is  
coercive.
\end{proof}

\begin{thm} \label{thm3.6} 
The following assertions hold.

$a)$
The unit contraction $C$-operates on $\mbox{\rm\sf L}_p$.

$b)$
The contraction $T_{-}$  $C$-operates on $\mbox{\rm\sf L}_p$, where $T_{-} u:= u^-$.

$c)$
Each $C^1$ normal contraction operates  on $\mbox{\rm\sf L}_p$.
\end{thm}

\begin{proof}
$a)$ Let $u\in\cd_p$.
From Proposition 3.1 $c)$  we deduce that $T_1u=u^+\land 1\in\cd_p$
and  $(\ref{6.2})$ implies
$$
\Gamma(T_1u,w)=1_{[0< \tilde{u}< 1]}\cdot\Gamma(u,w)
\quad \mbox{ for all } \, w\in\cd(\ce).
$$
Taking into account this equality, we obtain for all
$u,v\in\cd_p$:\\[1mm]
$ \langle \mbox{\sf L}_p(v+T_1 u), u-T_1 u \rangle
-  \langle \mbox{\sf L}_pv, u-T_1 u \rangle =
\int_X[\Gamma(v+T_ 1u)^\frac{p-2}{2}\Gamma
(v+T_1u,u-T_1 u)-\Gamma(v)^\frac{p-2}{2}\cdot\Gamma(v, u-T_1 u)]dm=
\int_X(1-1_{[0< \tilde{u}< 1]})\cdot
[\Gamma(v+T_1 u)^\frac{p-2}{2}\cdot
\Gamma(v+T_1 u, u)-\Gamma(v)^\frac{p-2}{2}\cdot\Gamma(v, u)]dm=
\int_{[\tilde{u}\leq 0]\cup[\tilde{u}\geq 1]} 
[\Gamma(v+T_1 u)^\frac{p-2}{2}\cdot\Gamma(v+T_1 u, u)-
\Gamma(v)^\frac{p-2}{2}\cdot\Gamma(v, u)]dm =0$. 
The last equality holds since
$T_1 \tilde{u}=0$ on
$[\tilde{u}\leq 0]$
and
$T_1 \tilde{u}=1$
on
$[\tilde{u}\geq 1]$. 
Therefore, using also
 $(\ref{7.2})$ we get
$\Gamma(v+T_1 u)=\Gamma(v)$,
$\Gamma(v+T_1 u,u)=\Gamma(v,u)$, on
$[\tilde{u}\leq 0]\cup[\tilde{u}\geq 1]$,
$$
\Gamma(v+T_1 u)^\frac{p-2}{2}\cdot\Gamma(v+T_1 u, u)=
\Gamma(v)^\frac{p-2}{2}\cdot\Gamma(v,u)
$$
on the same set.
We conclude that the unit contraction $C$-operates. \\

\noindent
$b)$ Again from  Proposition  \ref{prop3.1}  b) and $(\ref{6.2})$ 
we deduce that
$T_{-}u\in\cd_p$ for all $u\in\cd_p$
and
$$
\Gamma(T_{-}u,w)=\Gamma(u\land 0,w)=
1_{[\tilde{u}< 0]}\cdot\Gamma(u,w) \quad
\mbox{ for all } \, w\in\cd(\ce).
$$
On the set $[\tilde{u}\geq 0]$ we have $T_{-} \tilde{u}=0$ and therefore
$$
\Gamma(v+T_{-}u,w)=\Gamma(v,w)\mbox{ on } [\tilde{u}\geq 0].
$$
As in the proof of assertion a) we have:
$$
\langle \mbox{\sf L}_p(v+T_{-}u), u-T_{-}u  \rangle
- \langle \mbox{\sf L}_pv, u-T_{-}u \rangle  =\!\!
\int_X \!\! \left(1-1_{[\tilde{u}< 0]}\right)
[\Gamma(v+T_{-}u)^\frac{p-2}{2} \Gamma(v+T_{-}u ,u) -
\Gamma(v)^\frac{p-2}{2} \Gamma(v,u)]dm
$$
$$
=
\int_{[\tilde{u}\geq 0]}
[\Gamma(v+T_{-}u)^\frac{p-2}{2} \Gamma(v+T_{-}u,u)-
\Gamma(v)^\frac{p-2}{2}\cdot\Gamma(v,u)]dm=0.
$$
\\

\noindent
$c)$ Let $T\in C^1$ be a normal contraction.
It results $\vert T' \vert\leq 1$ and from assertion $2)$ of Proposition \ref{prop6.1}
we have 
$$
\Gamma(Tu,w)=T'(u)\Gamma(u,w) \quad
\mbox{ for all } \, u,w\in\cd(\ce).
$$
We deduce that for  $u\in\cd_p$ we have
$$
\Gamma(Tu)=( T'(u) )^2\Gamma(u)\leq\Gamma(u)
$$
and thus
$Tu\in\cd_p$.

We apply now $(\ref{3.6})$ for $u+Tu+v$ and $v$:
$$
\Gamma(u+Tu+v)^\frac{p-2}{2}\cdot\Gamma(u+Tu+v,u+Tu)-
\Gamma(v)^\frac{p-2}{2}\cdot\Gamma(v,u+Tu)\geq 0
$$
or equivalently
$$
(1+T^{'}(u))[\Gamma(u+Tu+v)^\frac{p-2}{2} \Gamma
(u+Tu+v,u)-\Gamma(v)^\frac{p-2}{2} \Gamma(v,u)]\geq 0.
$$
As a consequence, on the set $M:=[1 + T'(u)>0]$ we have
$$
\Gamma(u+Tu+v)^\frac{p-2}{2}\cdot\Gamma(u+Tu+v, u)-
\Gamma(v)^\frac{p-2}{2} \Gamma(v,u)\geq 0.
$$
Multiplying with the positive function 
$1-T'(u)$ and taking into account the equality
$$
(1-T'(u))\Gamma(w,u)=\Gamma(w,u-Tu)
$$
we obtain on $M$
\begin{equation} \label{3.7}
\Gamma(u+Tu+v)^\frac{p-2}{2}\cdot\Gamma(u+Tu+v,u-Tu)-
\Gamma(v)^\frac{p-2}{2}\Gamma(v,u-Tu)\geq 0.
\end{equation}
On the set $X\setminus M$ we have  $1+T'(u)=0$ and therefore
on this set we have the equalities
$$
\Gamma(u+Tu+v)=\Gamma(v), \ 
\Gamma(u+Tu+v,u-Tu)=\Gamma(v,u-Tu)
$$
that lead to
$$
\Gamma(u+Tu+v)^\frac{p-2}{2}\cdot\Gamma(u+Tu+v,u-Tu)-
\Gamma(v)^\frac{p-2}{2}\cdot\Gamma(v,u-Tu)=
$$
$$
\Gamma(v)^\frac{p-2}{2}\cdot\Gamma(v,u-Tu)-
\Gamma(v)^\frac{p-2}{2}\cdot\Gamma(v,u-Tu)=0
$$
on $X\setminus M$.
We conclude that the inequality $(3.7)$ holds on  $X$ and
integrating with  $m$ it results
$$
\langle \mbox{\sf L}_p(u+Tu+v), u-Tu  \rangle
-  \langle \mbox{\sf L}_pv, u-Tu \rangle  =
$$
$$
\int_X[\Gamma(u+Tu+v)^\frac{p-2}{2}\cdot
\Gamma(u+Tu+v,u-Tu)-
\Gamma(v)^\frac{p-2}{2}\cdot\Gamma(v,u-Tu)]dm\geq 0,
$$
and thus the contraction $T$ operates on $\mbox{\sf\rm L}_p$.
\end{proof}

\noindent
{\bf Remark.}
{\it For each $\alpha>0$  the contraction $T_\alpha$ $C$-operates 
on $\mbox{\rm\sf L}_p$ where $T_\alpha(u):=u^+\land\alpha$.} \\
Indeed, the assertion follows from
$T_\alpha u=\alpha T_1\left(\frac{u}{\alpha}\right)$, 
since the contraction
$T_1$ $C$-operates (cf. Theorem  \ref{thm3.6} a)) and using the homogeneity of the 
operator
$\mbox{\sf\rm L}_p$ (see Proposition \ref{prop3.4}).

\begin{prop} \label{prop3.7}
The operator $\mbox{\sf\rm L}_p$ is hemi-continuous.
\end{prop}

\begin{proof}
We have
$$
\langle \mbox{\sf L}_p(v+t(u-v)), u-v \rangle  =
\int_X\Gamma(tu+(1-t)v)^\frac{p-2}{2}
\Gamma(tu+(1-t)v, u-v) dm.
$$
If  $t\in [0,1]$ then:
$$
\Gamma (tu+(1-t)v)^\frac{p-2}{2}\cdot\Gamma(tu+(1-t)v,u-v)=
$$
$$
[t^2\Gamma(u)+2t(1-t)\Gamma(u,v)+
(1-t)^2\Gamma(v)]^\frac{p-2}{2} 
[t\Gamma(u,u-v)+(1-t)\Gamma(v,u-v)]\leq
$$
$$
[\Gamma(u)+2\Gamma(u)^\frac{1}{2}\Gamma(v)^\frac{1}{2}+\Gamma(v)]^\frac{p-2}{2}
\Gamma(u-v)^\frac{1}{2}
[\Gamma(u)^\frac{1}{2}+\Gamma(v)^\frac{1}{2}]\in
L^1(X,m).
$$
By dominated convergence we get the requested continuity.
\end{proof}

\begin{thm} \label{thm3.8}
Let  $(\ce,\cd(\ce))$ be regular strongly local Dirichlet form,
admitting a carr\'{e} du champ operator.
If $(\ce,\cd(\ce))$ is coercive 
then the associated nonlinear
$p$-form  is a nonlinear  Dirichlet form.
\end{thm}

\begin{proof}
By assertion $2)$ of Proposition \ref{prop3.2} we get that $\cd_p$ is reflexive.
Proposition \ref{prop3.7}  and Remark 3 at page 14 in \cite{Beu} 
imply now   that the operator 
$\mbox{\sf L}_p: \cd_p\longrightarrow \cd_p'$
satisfies the Browder property. 
We have shown (in Theorem \ref{thm3.6} b)) 
that the contraction 
$T_{-}$ $C$-operates on $\mbox{\sf L}_p$.
By Lemma 2.3.3  from \cite{Beu} we deduce that property $D1$ is verified. 
Since the contractions $T_\alpha$, $\alpha\in\R_+$,  $C$-operate, 
applying Lemma $2.3.6$ from \cite{Beu} we deduce that property $D2$ also holds.
The strict monotonicity and coercivity of the operator $L_p$ 
derive from Theorem \ref{thm3.5}. 
\end{proof}

\begin{cor} \label{cor3.9}
The following assertions hold.

$a)$ The form associated with the $p$-Laplace operator is  
a nonlinear Dirichlet form on  $W_0^{1,p}(\Omega)$.

$b)$ The nonlinear $p$-form associated   with
the Dirichlet  form given by    the uniformly elliptic case is  a
nonlinear  Dirichlet form on $W_0^{1,p}(\Omega)$.
\end{cor}

\begin{proof}
It suffices to prove assertion $b)$.
Since $\cd_p : =(W_0^{1,2}(\Omega))_p$ is a closed subspace of $W^{1,p}(\Omega)$ 
that is reflexive, 
it follows that $\cd_p$ is also reflexive.
The coercivity of the linear form is a consequence of the Poincar\'e 
inequality for functions from $W_0^{1,2}(\Omega)$ 
(see e.g. \cite{HKM}), using the uniform ellipticity.
By Theorem \ref{thm3.8} we deduce that the associated 
$p$-form is a nonlinear Dirichlet form on $\cd_p$. 
The restriction of $\mbox{\sf L}_p$ to $W_0^{1,p}(\Omega)\subset\cd_p$ is 
therefore strictly monotone and coercive. 
If a contraction $T$ $C$-operates or operates on 
$(\Delta_p,\cd_p)$,  then the same property is 
transmitted to the restriction of 
$\Delta_p$ to $W_0^{1,p}(\Omega)$, since
$Tu\in\cd_p$ whenever $u\in W_0^{1,p}(\Omega)$ and from 
$|Tu|\leq u$ we deduce that $Tu\in W_0^{1,p}(\Omega)$. 
We conclude that we have obtained a nonlinear Dirichlet 
form on $W_0^{1,p}(\Omega)$.
\end{proof}

\noindent
{\bf  Capacity and equilibrium potential.} 
Let  $(\ce,\cd(\ce))$ be a regular strongly local Dirichlet form,
admitting a carr\'{e} du champ operator $\Gamma$ and such that 
$(\ce,\cd(\ce))$ is coercive.

We follow the approach from \cite{HKM}, Theorem 2.2,  to construct the Choquet capacity induced by the nonlinear $p$-form  $(\ce^p, \cd_p)$ associated with $(\ce,\cd(\ce))$.

For a compact subset $K$ of $X$ let $W(K):= \{ w\in \cd_p\cap C(X) : w\geqslant 1 \mbox{ on }  K \} $ 
and define
 $$
 {\sf cap}^o_p (K):= \inf \{ \langle L_p u,  u \rangle : u\in W(K) \}\, 
 \mbox{ provided that }W(K)\not= \emptyset
 $$ 
  and
${\sf cap}^o_p (K):=\infty$ if  $W(K)= \emptyset$.   
Recall that we have 
  $ \langle L_p u,  u \rangle  =  \int_X  \Gamma (u)^{ \frac{p}{2} }dm$.

If $U\subset X$ is open,  set
$$
{\sf cap}_p (U):=\sup\{ {\sf cap}^o_p (K): K\subset U \, \mbox{ compact} \}
$$
and finally, for an arbitrary set $E \subset X$ 
$$
{\sf cap}_p (E):=\inf\{ {\sf cap}_p (U):  U \mbox{ open}, E\subset  U \}.
$$

We  show now that there is no ambiguity for the  definition of the capacity of a compact set and we prove the existence of the equilibrium potential for a compact set.

\begin{prop} \label{prop3.10} 
The following assertions hold for a compact subset $K$ of  $X$.

$1)$ We have ${\sf cap}^o_p (K)= {\sf cap}_p (K)$.

$2)$ There exists a unique function $e_K \in W(K)$ such that
$$
 \langle L_p e_K , w- e_K \rangle \geq 0 \mbox{ for all } w\in W(K) 
 $$
 and we have $0\leq e_U\leq 1$, $e_K=1$  on $K$, and 
$$
{\sf cap}_p(K)= \langle L_p e_K , e_K \rangle =\int_X \Gamma(e_K)^\frac{p}{2} dm. 
$$
\end{prop}
\begin{proof}
$1)$ The inequality ${\sf cap}^o_p (K)\leq  {\sf cap}_p (K)$ is clear. 
To prove the converse inequality let $\varepsilon >0$, $u\in W(K)$, and consider the open set 
$U_\varepsilon:= \{ x\in X: u > 1-\varepsilon \}$. 
Then $K\subset U_\varepsilon$ and let $L$ be a compact subset of $U_\varepsilon$, such that $K\subset L$ and
${\sf cap}_p (U_\varepsilon)\leq {\sf cap}^o_p (L)+\varepsilon$. 
We have $\frac{1}{1-\varepsilon} u \in W(L)$ and therefore ${\sf cap}^o_p (L)
\leq \int_X \Gamma ( \frac{1}{1-\varepsilon} u)^{ \frac{p}{2} }dm$.
We conclude that 
${\sf cap}_p (U_\varepsilon) \leq    (1- \varepsilon)^{-p}  \int_X \Gamma (  u)^{ \frac{p}{2} }dm +\varepsilon$.
Taking the infimum over  all $u\in W(K)$ we get
${\sf cap}_p (K) \leq {\sf cap}_p (U_\varepsilon)  \leq (1- \varepsilon)^{-p}  {\sf cap}^o_p (K) +\varepsilon$
and letting $\varepsilon \longrightarrow 0$ we obtain the claimed inequality.

$2)$  The existence of the function $e_K$ follows by $(\ref{2})$ because  assertion $b)$ of Theorem \ref{thm3.5} implies that the operator $L_p$ is strictly monotone.
The properties $0\leq e_K\leq 1$ and $e_K=1$  on $K$ follow as in the proof of Theorem 6.2.1 from \cite{Beu}.
Arguing as in the proof of Lemma 6.2.1  from \cite{Beu} and using the sector condition (assertion $b)$ of Proposition $\ref{prop3.3}$), we obtain
the equality ${\sf cap}_p(K)= \langle L_p e_K , e_K\rangle$.
\end{proof}

We show now that ${\sf cap}_p$ is a Choquet capacity on $X$. 
The next theorem is a generalisation, in our frame given by a regular strongly local Dirichlet form, of a result for the $p$-Laplace operator, Theorem 2.2 from \cite{HKM}.

\begin{thm} \label{thm3.11}
The set function $E\longmapsto {\sf cap}_p (E)$, $E\subset X$, enjoys the following properties.

$1)$  ${\sf cap}_p$ is strongly subadditive on compacts sets, that is, if if $K, L$ are compact subsets of $X$ then
$$
{\sf cap}_p (K\cup L) + {\sf cap}_p(K\cap L)\leq {\sf cap}_p(K) + {\sf cap}_p (L).
$$

$2)$ The set function $E\longmapsto {\sf cap}_p (E)$ is a Choquet capacity on $X$, that is, the following properties hold:

\noindent
$(2a)$ If $E_1\subset E_2$ then ${\sf cap}_p(E_1)\leq {\sf cap}_p (E_2)$.

\noindent
$(2b)$ If $(K_i)_i$ is a decreasing sequence of compact subsets of $X$ with $K=\bigcap_i K_i$ then ${\sf cap}_p(K)=\lim_i {\sf cap}_p(K_i)$.

\noindent
$(2c)$  If $(E_i)_i$ is an  increasing sequence of subsets of $X$ with $E=\bigcup_i E_i$ then ${\sf cap}_p(E)=\lim_i {\sf cap}_p(E_i)$.

\vspace{1mm}
$3)$ If  $E=\bigcup_i E_i$ then ${\sf cap}_p(E) \leq \sum_i  {\sf cap}_p(E_i)$.

$4)$ If $E\subset X$ and   ${\sf cap}_p(E) = 0$ then the set $E$ is $m$-negligible. 
\end{thm}

\begin{proof}
To prove assertion $1)$ we argue as in  \cite{Beu}, page 71.
From $(\ref{6.4})$ we get 
$\int_X \Gamma(u\wedge v)^{\frac p 2} dm +  \int_X \Gamma(u\vee v)^{\frac p 2} dm = \int_X  \Gamma(u)^{\frac p 2} dm + \int_X  \Gamma(v)^{\frac p 2}  dm$.
This implies that if $u\in W(K)$ and $v\in W(L)$ then $u\wedge v\in W(K\cap L)$, $u\vee v \in W(K \cup L)$ and therefore
$$
{\sf cap}_p (K\cup L) + {\sf cap}_p(K \cap L) 
\leq  \int_X \Gamma(u\wedge v)^{\frac p 2} dm +  \int_X \Gamma(u\vee v)^{\frac p 2} dm =
\int_X  \Gamma(u)^{\frac p 2} dm + \int_X  \Gamma(v)^{\frac p 2}  dm.
$$
Taking now the infimum over  $W(K)$ and  $W(L)$ we get the strong subadditivity condition for $K$ and $L$.

Clearly, property $(2a)$ is an immediate consequence of the definition.

$(2b)$ We argue as in the proof of assertion $1)$ of Proposition \ref{prop3.10}. 
Let $\varepsilon >0$, $u\in W(K)$, and 
$U_\varepsilon:= \{ x\in X: u > 1-\varepsilon \}$. 
Then $K\subset U_\varepsilon$ and there exists $i_o$ such that  $K_{i_o} $ lies in the open set $U_\varepsilon$.
Consequently,  $\frac{1}{1-\varepsilon} u \in W(K_{i_o})$ and therefore, using also assertion $1)$ of Proposition  \ref{prop3.10}, we have 
${\sf cap}_p (K_{i_o}) = {\sf cap}^o_p (K_{i_o})\leq     (1- \varepsilon)^{-p}  \int_X \Gamma (  u)^{ \frac{p}{2} }dm$
and letting $\varepsilon \longrightarrow 0$ we get
$\lim_i {\sf cap}_p (K_i)\leq   \int_X \Gamma (  u)^{ \frac{p}{2} }dm$ for all $u\in W(K)$.
Taking the infimum over  all $u\in W(K)$ we get
${\sf cap}_p (K) \leq \lim_i {\sf cap}_p (K_i) \leq {\sf cap}_p (K)$.

To prove assertions $(2c)$  and $3)$ we need the following lemma, which is an adaptation  of Lemma 2.3 from \cite{HKM} to our more general frame. 
For the reader convenience we present its proof in Appendix {\bf A3}.

\begin{lem} \label{lem3.12}
Let $E_1,\ldots , E_k \subset X$ and  $F_i\subset E_i$, $i=1, \ldots , k$,  such that 
${\sf cap}_p (\bigcup_{i=1}^k F_i ) <\infty$. 
Then
\begin{equation} \label{3.8}
{\sf cap}_p (\bigcup_{i=1}^k E_i ) - {\sf cap}_p (\bigcup_{i=1}^k F_i ) \leq
\sum_{i=1} ^k ( {\sf cap}_p (E_i ) - {\sf cap}_p (F_i )).
\end{equation}
\end{lem}

$(2c)$ By the monotonicity property  $1)$ we only have to prove that 
${\sf cap}_p(E)\leq \lim_i {\sf cap}_p(E_i) $ and we may assume that ${\sf cap}_p(E_i) <\infty$ for all $i$.
Let $\varepsilon>0$ be fixed and for each $i$ choose an open set $U_i$ such that $E_i\subset U_i$
and 
${\sf cap}_p (U_i ) \leq {\sf cap}_p (E_i ) + \frac{\varepsilon}{2^i} $. 
Because 
${\sf cap}_p (\bigcup_{i=1}^k E_i )= {\sf cap}_p (E_k )<\infty$ for each $k$,  by Lemma \ref{lem3.12} we get
${\sf cap}_p (\bigcup_{i=1}^k U_i ) - {\sf cap}_p (\bigcup_{i=1}^k E_i ) \leq
\sum_{i=1} ^k  \frac{\varepsilon}{2^i} ,\varepsilon.$
If $K\subset \bigcup_{i=1}^\infty U_i$ is compact, then $K\subset \bigcup_{i=1}^k U_i$ for some $k$ and we have
${\sf cap}_p(K) \leq {\sf cap}_p(  \bigcup_{i=1}^k U_i) \leq {\sf cap}_p(  \bigcup_{i=1}^k E_i) +\varepsilon
\leq \lim_k {\sf cap}_p( E_k) +\varepsilon$.
We conclude that
${\sf cap}_p(E) \leq {\sf cap}_p (\bigcup_{i=1}^\infty U_i )= \sup\{ {\sf cap}_p (K): K \mbox{ compact, } K\subset \bigcup_{i=1}^\infty U_i \}
\leq \lim_k {\sf cap}_p(E_k) + \varepsilon$ 
and the proof of $(2c)$ is complete.

$3)$ Lemma \ref{lem3.12} implies that the finite version of $3)$ holds. 
Because $\bigcup_{i=1}^kE_i$ increases to  $\bigcup_{i=1}^\infty E_i$, we may  applying $(2c)$ to obtain  $3)$.

$4)$ Because ${\sf cap}_p(E) = 0$,  there exists a Borel set $E_o$ such that $E\subset E_o$ and ${\sf cap}_p(E_o) = 0$.
So, replacing $E$ by $E_o$, we may assume that $E\in \cb(X)$.
Recall that by the Choquet capacitability theorem any Borel set (actually, any analytic set) is capacitable, that is, 
${\sf cap}_p(E)= \sup\{ {\sf cap}_p (K): K \mbox{ compact, } K\subset E \}$.
Therefore, we may assume that $E$ is compact and in this case, by assertion $1)$ of Proposition $\ref{prop3.10}$,  we have
${\sf cap}_p(E)={\sf cap}^o_p(E)$.
Let $u\in W(E)$, then since $u\geq 1$ on $E$ and by assertion $b)$ of Theorem $\ref{thm3.5}$ we have
$m(E)\leq \| u \|^p_{\cd_p} \leq c \cdot \langle {\sf L}_p u, u\rangle$.
We conclude that
$m(E)\leq c \cdot
\inf\{  \langle {\sf L}_p u, u\rangle: u \in W(E) \}= {\sf cap}^o_p(E)=0.$
\end{proof}

\begin{rem} \label{rem3.13}   
$1)$ The function $e_K$  from assertion $2)$ of Proposition \ref{prop3.10} is called {\rm equilibrium potential}.

$2)$ Assertion $4)$ of Theorem \ref{thm3.11} and its proof are  standard,
the coercivity of the form $\ce$ is used essentially here.
For the case of the $p$-Laplace operator see Lemma 2.10 from \cite{HKM}.

$3)$  For convenience we assumed that $(\ce,\cd(\ce))$ is a regular Dirichlet form on a locally compact separable metric space $X$.
However, the results from this section may be extended to the quasi-regular case (cf. \cite{MR}), on a general Lusin topological space $X$.

$4)$ As we mentioned at the end of the Introduction, an  independently achieved  result on the $p$-energy forms and the induced capacity is contained  in \cite{K}. 
However, our $p$-form is closer to the classical situation. Indeed, recall that in $(\ref{p-gradient})$ we succeeded to expressed the $p$-form  $\ce^p$ by means of a gradient operator. 
In this way, comparing $(\ref{p-gradient})$ with $(\ref{3.5})$, we emphasised that $\ce^p$  is a generalisation of the $p$-form associated with the $p$-Laplace operator.
The capacity in \cite{K} is constructed starting with the open sets and the existence of the equilibrium potentials.
We constructed the capacity starting with compacts, following the approach from \cite{HKM} for the $p$-Laplace operator.
We also proved the existence of the equilibrium potentials  following \cite{Beu}, (see Proposition \ref{prop3.10}). 
Note that we need continuous functions from $\cd_p$ to define the capacity, therefore, in general our capacity might be  different from the capacity considered in \cite{K}.

$5)$ Energy forms on $L^p$-spaces and associated capacities have been studied  in \cite{HJ} and \cite{JS}, 
using the $\Gamma$-transform of a given $L^p$ sub-Markovian semigroup.
\end{rem}

\section{The nonlinear Dirichlet $p$-form  associated with a quasiregular mapping} \label{sect4}. 

In the previous section
it was proved that to each regular symmetric strongly local (linear) 
Dirichlet form 
$(\ce,\cd(\ce))$ on $\mbox{\sf L}^2(X,m)$, admitting a carr\'e du champ. $\Gamma$, 
and every real number $p>1$, one can associate a nonlinear form 
$\ce^p$ by the formula
$\ce^p(u,v)=\int_X\Gamma(u)^\frac{p-2}{2}\Gamma(u,v)dm,
$
where $u,v\in\cd_p:=\{w\in\cd(\ce)\cap
\mbox{\sf L} ^p(X,m): \Gamma(w)^\frac{1}{2}
\in\mbox{\sf L}^p(X,m)\}$.
It turns out that $(\ce^p,\cd_p)$ is a nonlinear Dirichlet form in the sense 
introduced by P. van Beusekon  in \cite{Beu}.

Further on,  using the above procedure, 
we associate with each $n$-dimensional quasiregular mapping 
$f$ a nonlinear Dirichlet form $\ce^n$ 
$(p=n)$, such that the components of $f$ become harmonic functions with respect to
$\ce^n$.

We prove in Section \ref{sect6} below Caccioppoli type inequalities on balls in the intrinsic metric on $X$
generated by a general Dirichlet form $\ce$ as above, for functions that are harmonic with respect to $\ce^p$.
We deduce in this way the usual Caccioppoli type inequality in $\R^n$
(cf. [BI]).

Let $(\ce,\cd(\ce))$ be a symmetric regular Dirichlet form on 
$\mbox{\sf L}^2(X,m)$ 
that is strongly local and admits a carr\'e du champ operator $\Gamma$
(cf. [FOT] and [BH2]). If $U$ is an open subset of $X$ then we put
$$
\cd_{p}|_U:=\{ u:U\longrightarrow\R:  \mbox{ there exists }  \bar{u}\in\cd_p, \, u=\bar{u} \;
m-\mbox{a.e.} \mbox{ on } U\}
$$
$$
\cd_p(U)_c:=\{v\in\cd_p : \pp v\subset U \mbox{ compact}\}
$$
For every $u\in\cd_p|_U$ and $v\in\cd_p(U)_c$ we define
$$
\ce^p(u,v):=\ce^p(\bar{u},v),
$$
where
$\overline{u}\in\cd_p$ and $u=\overline{u}$ 
$m$-a.e. on $U$.
Since  $\Gamma(u)=0$ on  $U$ provided $u=0$ $m$-a.e. on  $U$,
we conclude that the above definition is correct.

We set:
$$
(\cd_p)_l:=\lbrace f: X\longrightarrow\R:   \mbox{ for all }  x\in X
\mbox{ there exists } V \mbox{ open with  } x \in V\mbox{ and }
f{\vert_V} \in\cd_{p}\vert_V \rbrace .
$$

A function  $u\in(\cd_p)_{l}$ 
is called  $\mbox{\sf L}_p$-{\it harmonic on an open set } $U$ provided that 
$\ce^p(u, v) = 0
$
for each open set  $V\subset U$ such that
$u_{\vert_V}\in\cd_{p}\vert_V$ \c and all $v\in\cd_p(V)_c$.

Let  $\Omega$ be a domain in $\R^{p}$ and
${\cal A}: \Omega\times\R^{p}\longrightarrow\R^{p}$ 
defined by
\begin{equation} \label{4.1}
{\cal A}(x,\xi )
:=(G(x)\xi,\xi)^{\frac{p-2}{2}}G(x)\xi, 
\end{equation}
where  $G(x)$ is  a symmetric positive definite  $p\times p$-matrix 
of measurable functions on  $\Omega$,
and there exists two strictly positive constants  $\alpha$ and 
$\beta$ such that the following condition of uniform ellipticity holds:
\begin{equation} \label{4.2}
\alpha|\xi|^{2}\leq (G(x)\xi,\xi)\leq\beta|\xi|^{2} \, \mbox{ for all } x\in\Omega \mbox{ and } \xi\in\R^{p}.
\end{equation}

A function $u\in W_{loc}^{1,p}(\Omega)$  is called
{\it  weak solution of the equation}
$$
 \di {\cal A}(x,\nabla u(x))=0
$$
provided that 
$$
\displaystyle\int\limits_{\Omega}
({\cal A}(x,\nabla u(x)), \nabla v(x)) dx = 0
\quad \mbox{ for all }  \, v\in C_{c}^{\infty}(\Omega).
$$

\vspace{-2mm}

The solutions that are continuous functions are named 
{\it ${\cal A}$-harmonic} (see [HKM]).

\begin{rem} \label{rem4.1} 
The uniform ellipticity condition $(\ref{4.2})$ allows
the construction, for each  ${{\cal A}}$,
of a linear  regular Dirichlet form, induced by the matrix 
$G$  and further of the nonlinear  
$p$-form
$(\ce^p,\cd_p)$ associated 
$\ce^p(u,v)=  \langle \mbox{\sf L}_p u, v  \rangle$, a
nonlinear Dirichlet form.
Note that if $u\in W_{loc}^{1,p}(\Omega)$, 
then  $u\in(\cd_p)_l$; cf. Theorem 1.2.2 in [AH].
From  the  expression of the  carr\'{e} du champ operator 
in this particular case, we deduce 
that an ${{\cal A}}$-harmonic function is 
$\mbox{\sf L}_p$-harmonic on $X$.
\end{rem}

{\it In the sequel we fix a quasiregular mapping
$f:\Omega\longrightarrow\R^{p}$;} 
see Appendix {\bf A4} for the definition and basic facts on the quasiregular mappings.

We consider the matrix $\theta_{f}$,  given by $(\ref{7.10})$,  associated with $f$.
Since $f$ is quasiregular, 
the uniform ellipticity condition $(\ref{7.11})$ holds.
It follows that that 
$G:=\theta_{f}$
is a symmetric positive definite   $p\times p$-matrix of measurable functions
on  $\Omega$, satisfying the uniform ellipticity condition.
Therefore,  by $(\ref{4.1})$ we may define  ${\cal A}$, depending on $f$.

The following result may be deduced from Theorem 5.1 in \cite{Re2}; see also 
\cite{Be} and \cite{HKM}.

\begin{prop} \label{prop4.2}
The components $f^i$, $1\leq i\leq p$, of $f=(f^1,...,f^p)$ are $\ca$-harmonic
functions.
\end{prop}

\begin{rem} \label{rem4.3}
A second example of ${\cal A}$-harmonic function is $\ln |f|$, where 
$f:\Omega\rightarrow\R^n\setminus\{0\}$ is quasiregular (cf. [BI] or [Re1]).
\end{rem}

\begin{thm} \label{thm4.3}
Let  $\Omega\subset\R^p$ be a bounded domain and
$f:\Omega\longrightarrow\R^p$  a quasiregular mapping.
Then there exists a nonlinear Dirichlet $p$-form 
$(\ce^p,\cd_p)$ such that the components of  
$f$ as well as  $\ln\vert f\vert$, whenever
$f:\Omega\longrightarrow\R^p\setminus\{0\}$, 
are $\mbox{\rm \sf L}_p$-harmonic functions, 
where  $\mbox{\rm \sf L}_p$ is the nonlinear operator
generating the form. 
\end{thm}
\begin{proof}
The assertion follows from Remark \ref{rem4.1}, Proposition \ref{prop4.2},  
and Remark \ref{rem4.3} (for the function $\ln|f|$).
\end{proof}

\section{Caccioppoli type inequlities} \label{sect6}

We obtain Caccioppoli type inequalities in the
general context given by a Dirichlet form and then we 
apply these results to the quasiregular mappings.

{\it In this section $(\ce,\cd(\ce))$ 
is a strongly local regular Dirichlet form, admitting a
carr\'{e} du champ opeartor.}

\begin{prop} \label{prop5.1}
Let  $U\subset X$ be an open set,
$u\in\cd_p\vert_U$ a
$\mbox{\rm\sf L}_p$-harmonic function on $U$ and 
$\varphi\in\cd(\ce)\cap L^\infty(X,m)$
having compact support in $U$, $\varphi\geq 0$ 
and $\Gamma(\varphi)$  bounded.
Then for each  $c\in\R$ we have
$$
\left(\displaystyle\int\limits_X
\varphi^p\Gamma(u)^\frac{p}{2}dm\right)^\frac{1}{p}
\leq p\left(
\displaystyle\int\limits_X
\Gamma(\varphi)^\frac{p}{2}\vert u-c\vert^pdm\right)
^\frac{1}{p}.
$$
\end{prop}

\begin{proof}
Let us first remark that, since $\varphi$  and 
$\Gamma(\varphi)$ are bounded functions,
we  deduce that for each $\beta>0$ we have
$$
\varphi^\beta\in\cd_p(G)_c.
$$
Let $\eta:=\varphi^p(u-c)$
and
$\eta_n:=\varphi^p(u_n-c)$ where
$u_n:=(u\land n)\lor (-n)$.
It is known that $u_n\longrightarrow u$ in $\cd(\ce)$  and from Proposition \ref{prop3.1}
we deduce that $\eta_n\in\cd_p(G)_c$.
The function $u$ being $\mbox{\sf L}_p$-harmonic on $U$ and applying 
the chain rule, we get
$$
0=\ce^p(u,\eta_n)= \!\!
\displaystyle\int\limits_X
\Gamma(u)^\frac{p-2}{2}\Gamma(u,\varphi^p(u_n-c))dm=\!\!
\displaystyle\int\limits_X
\Gamma(u)^\frac{p-2}{2}\varphi^p\Gamma(u,u_n-c)dm+
\displaystyle\int\limits_X
\Gamma(u)^\frac{p-2}{2}(u_n-c)\cdot p\varphi^{p-1}\Gamma(u,\varphi)dm
$$
$$
=\displaystyle\int\limits_X
\Gamma(u)^\frac{p-2}{2}\varphi^p\Gamma(u, u_n)dm +
p\displaystyle\int\limits_X
\varphi^{p-1}(u_n-c)\Gamma(u)^\frac{p-2}{2}\cdot\Gamma(u,\varphi)dm.
$$
By the  H\"{o}lder inequality we obtain now
$$
\displaystyle\int\limits_X
\Gamma(u)^\frac{p-2}{2}\varphi^p\Gamma(u,u_n)dm\leq 
p\displaystyle\int\limits_X
\varphi^{p-1}\vert u_n-c\vert\Gamma(u)^\frac{p-2}{2}\Gamma(u,\varphi)dm\leq
p\displaystyle\int\limits_X
\Gamma(u)^\frac{p-2}{2}\varphi^{p-1}\vert u_n-c\vert
\Gamma(u)^\frac{1}{2}\Gamma(\varphi)^\frac{1}{2}dm=
$$
$$
p\displaystyle\int\limits_X
\Gamma(u)^\frac{p-1}{2}\varphi^{p-1}\vert u_n-c\vert\Gamma(\varphi)^
\frac{1}{2}dm\leq
p\left(\displaystyle\int\limits_X
(\Gamma(u)^\frac{p-1}{2}\cdot\varphi^{p-1})^\frac{p}{p-1}dm
\right)^\frac{p-1}{p}\cdot\left(
\displaystyle\int\limits_X
\vert u_n-c\vert^p\Gamma(\varphi)^
\frac{p}{2}dm\right)^\frac{1}{p}=
$$
$$
p\left(\displaystyle\int\limits_X
\Gamma(u)^\frac{p}{2}\varphi^pdm\right)^\frac{p-1}{p}\cdot
\left(\displaystyle\int\limits_X
\Gamma(\varphi)^\frac{p}{2}\vert u_n-c\vert^pdm\right)^\frac{1}{p}.
$$
Therefore we have
$$
\displaystyle\int\limits_X
\Gamma(u)^\frac{p-2}{2}\varphi^p\Gamma(u,u_n)dm\leq
p\left(\displaystyle\int\limits_X
\Gamma(u)^\frac{p}{2}\varphi^pdm\right)^\frac{p-1}{p}\cdot
\left(\displaystyle\int\limits_X
\Gamma(\varphi)^\frac{p}{2}\vert u_n-c\vert^pdm\right)^\frac{1}{p}.
$$
Passing to the limit, $u_n\longrightarrow u$ \ii n
$\cd(\ce)$ and 
$L^p(X,m)$, we conclude that 
$$
\displaystyle\int\limits_X
\Gamma(u)^\frac{p}{2}\varphi^pdm\leq p\left(
\displaystyle\int\limits_X
\Gamma(u)^\frac{p}{2}\varphi^pdm\right)^\frac{p-1}{p}\cdot
\left(
\displaystyle\int\limits_X
\Gamma(\varphi)^\frac{p}{2}\vert u-c\vert^pdm\right)^\frac{1}{p}
$$
and consequently we get the required inequality.
\end{proof}

\noindent
{\bf Definition.}
Let  $D \subset U$ be an open relatively compact subset of $U$.
We say that $D$ {\it admites a truncation function} in $U$
if there exists  $\varphi\in \cd(\ce)$,
having the following properties:\\
$0\leq\varphi\leq 1$,
$\Gamma(\varphi)$ bounded, $\varphi$ has
compact support in $U$ and $\varphi=1$ on  $\overline{D}$.

\begin{thm}\label{thm5.2}
Let $U$ be an open subset of  $X$, $u\in\cd_p\vert_U$ a 
$\mbox{\rm\sf L}_p$-harmonic function on  $U$ and
$D\subset U$ an open relatively compact subset of $U$, 
admitting a truncation function $\varphi$ in $U$.
Then for each $c\in\R$ we have
$$
\left(\displaystyle\int\limits_D
\Gamma(u)^\frac{p}{2}dm\right)^\frac{1}{p}\leq p\cdot k
\left(
\displaystyle\int\limits_F
\vert u-c\vert^pdm\right)^\frac{1}{p}
$$
where
$F=\pp\varphi$ \c si $\Gamma(\varphi)\leq k^2$.
\end{thm}

\begin{proof}
Applying Proposition \ref{prop5.1} to the truncation function $\varphi$ 
we get
$$
\displaystyle\int\limits_D
\Gamma(u)^\frac{p}{2}dm=
\displaystyle\int\limits_D
\varphi^p\Gamma(u)^\frac{p}{2}dm
\leq\displaystyle\int\limits_X
\varphi^p\Gamma(u)^\frac{p}{2}dm\leq
p^p\displaystyle\int\limits_F
\Gamma(\varphi)^\frac{p}{2}\vert u-c\vert^pdm\leq (p\cdot k)^p
\displaystyle\int\limits_{F}
\vert u-c\vert^pdm,
$$
that is the desired equality.
Note that we have used the fact that $\Gamma(\varphi)=0$ on $X\setminus F$.
\end{proof}

Further we use the intrinsic metric $\rho$ on $X$, induced by the Dirichlet form   $(\ce,\cd(\ce))$.
This technique has been developed in the works \cite{BM1}, \cite{BM2}, \cite{St1}, and \cite{St2}.
For each  $x,y\in X$  define the map 
$\rho :X \times X\longrightarrow[0,\infty]$  by
$$
\rho(x,y):= \sup\{u(x)-u(y) :  u\in\cd_{loc}\cap C(X), \; \mu_{\langle u\rangle }\leq m\}.
$$
One can verify that  $\rho(x,y)\leq\rho(x,z)+\rho(z,y)$ 
for each  $x,y,z\in X$.       
$\rho$ is called the {\bf intrinsic metric} induced by  $(\ce,\cd(\ce))$.
We denote by $B_r(x)$ the ball with radius $r>0$ centered in 
$x\in X$,
$$
B_{r}(x):=\lbrace y\in X: \rho(x,y)<r\rbrace,
$$
and let $\tau_{\rho}$ be the topology generated by these balls.

In the sequel we suppose that:\\
{\it the initial topology of locally compact space on $X$ is connected
and  coincides  with \mbox{$\tau_{\rho}.$}}

Note that under the above hypothesis $\rho$ becomes a metric on $X$.

For each  $x\in X$ \c and $r>0$   we  define the functions 
$\rho_{x}, \rho_{x,r}:X\longrightarrow\R_{+}$ by
$$
\rho_{x}(y):=\rho(x,y)\ ,  \quad 
\rho_{x,r}(y):=(r-\rho(x,y))\lor 0 \quad \mbox{ for all } \, y\in X.
$$
Then the functions $\rho_{x}$ and $\rho_{x,r}$ 
have the following properties  (see Lemma 1' in [St1]):
$$
\rho_{x}\in\cd_{loc}\cap C(X),
\quad\rho_{x,r}\in \cd\cap C_{c}(X), 
\mbox{ provided that } B_{r}(x) 
\mbox{ is relatively} \mbox{ compact and}
$$
$$ 
\mu_{\langle \rho_{x}\rangle }\leq m  \ , \quad \mu_{\langle \rho_{x,r}\rangle }\leq m.
$$

\begin{cor} \label{cor5.3}
Let $U$ be open, $x_0\in U$, $0< r< R$
such that the ball $\overline{B}_R(x_0)$ is included in  $U$.
If $u\in\cd_{p}\vert_U$ is $\mbox{\rm\sf L}_p$-harmonic 
on $U$ and  $c\in\R$, then
$$
\left(
\displaystyle\int\limits_{B_r(x_0)}
\Gamma(u)^\frac{p}{2}dm\right)^\frac{1}{p}
\leq\frac{p}{R-r}\cdot\left(
\displaystyle\int\limits_{B_R(x_0)}
\vert u-c\vert^pdm\right)^
\frac{1}{p}.
$$
\end{cor}

\begin{proof}
We define the function $\psi$ by
$$
\psi(y)=(R-\rho(x_0,y))_+\land(R-r)=\rho_{x_0,R}(y) \land (R-r) .
$$
It results that $\psi\in\cd(\ce)$ and $\Gamma(\psi)\leq 1$.
Defining $\varphi:=\frac{1}{R-r}\psi$
we have obtained a function possessing the following properties:
$
\varphi\in\cd(\ce)\cap L^\infty(X,m),$ 
$\Gamma(\varphi)\leq\frac{1}{(R-r)^2}$,
$0\leq\varphi\leq 1$,
$\varphi=1$ pe $\overline{B}_r(x_0)$ \c si $\varphi=0$ pe 
$X\setminus B_R(x_0)$.
We conclude that $\varphi$ is a truncation function for
 $B_r(x_0)$ in $U$ and  $\pp\varphi=\overline{B}_R(x_0)$.
Applying Theorem \ref{thm5.2} we obtain the required inequality. 
\end{proof}

\noindent
{\bf Final remarks.}
1. Proposition  \ref{prop5.1} may be applied to the Dirichlet form of the uniformly elliptic case. 
Taking into account the estimates for the carr\'{e} du champ operator in this case,
$\alpha\vert\nabla u\vert^2\leq
\frac{1}{2} \Gamma(u)\leq\beta\vert\nabla u\vert^2,$
we obtain
\begin{equation} \label{5.1}  
\left(
\displaystyle\int\limits_\Omega
\varphi^p(x)\vert\nabla u(x)\vert^pdx\right)^\frac{1}{p}
\leq
p \sqrt{ \frac{\beta}{\alpha} } 
\left(\displaystyle\int\limits_\Omega
\vert\nabla\varphi (x)\vert^p
\vert u(x)-c\vert^p\right)^\frac{1}{p}.
\end{equation}
This is a Caccioppoli type inequality precisely as in 
\cite{Ri} or
\cite{BI},
Proposition 6.1
(including the values of the constants).

2.  Corollary \ref{cor5.3}  gives a Caccioppoli type inequality as in \cite{BI},
however for balls in the intrinsic metric.

{3.}  Return to the Dirichlet form of the uniformly elliptic case and observe that the Euclidean ball of radius $r$ admits a truncation function $\varphi$ 
with $\vert\nabla\varphi\vert\leq\frac{1}{r}$ and 
$\pp\varphi$ included in the Euclidean ball of radius $2r$.
In addition, if $u\in W_{loc}^{1,n}(\Omega)$ then there exists an open set 
$G$ including $\pp\varphi$ and such that 
$u\in\cd_{p}\vert_G$; 
cf. \cite{AH}.
Applying now $(\ref{5.1})$ we deduce a Caccioppoli type inequality for Euclidean balls,  as in 
\cite{BI}, 
Corollary 6.1:
$$
\left(\displaystyle\int\limits_{\vert x_0-x\vert< r}
\vert\nabla u(x)\vert^pdx\right)^{\frac{1}{p}}
\leq
\frac{p}{r} \sqrt{ \frac{\beta}{\alpha} }
\left( \displaystyle\int\limits_{\vert x_0-x\vert< 2r}
\vert u(x)-c\vert^pdx \right)^\frac{1}{p}.
$$

4.  Let now $\Omega\subset\R^p$ be a bounded domain and 
$f:\Omega\longrightarrow \R^p$  a quasiregular mapping.
Since we observed that the component of $f$ as well as $\ln |f|$ 
(provided that $f:\Omega\longrightarrow\R^p\setminus\{0\}$)
are $L_p$-harmonic functions, we can apply  Proposition \ref{prop5.1}, 
Theorem \ref{thm5.2}, and Corollary \ref{cor5.3}, to obtain Caccioppoli type inequalities for the quasiregular mapping $f$.



\section{Appendix}

\noindent
{\bf A1. Dirichlet forms.}
Let $(X,\cb)$ be a measurable space and  $m$ a fixed  positive 
$\sigma$-finite measure on this space.

Let $(\ce,\cd(\ce))$ be a closed form on  $L^{2}(X,m)$.
Then there exists a unique selfadjoint operator 
$\mbox{\sf L}:\cd(\mbox{\sf L})\longrightarrow L^{2}(X,m)$,
such that
$\cd(\mbox{\sf L})\subset \cd(\ce)$ 
is dense in the norm  $\ce_{1}^{\frac{1}{2}}$, $-\mbox{\sf L}$
positive definite and
$$
\ce(u,v)=-(\mbox{\sf L}u,v)_{2}\quad
\mbox{ for all }\, u\in\cd(\mbox{\sf L})\mbox{ and }v\in\cd(\ce)
$$
(see e.g. \cite{Fu}).
The operator  $\mbox{\sf L}$ is called 
{\bf the generator of the form}  $(\ce , \cd ( \ce ) )$.
A normal contraction $T$ is said to {\bf operate } on $\vre$ 
provided that if 
$u\in\cd(\vre)$ then  $T(u)\in\cd(\vre)\mbox{ \c and }\vre(T(u),
T(u))\leq\vre(u,u).$

Let $\cd=\cd(\vre)$ be a dense linear subspace of $L^{2}(X, m)$.
A symmetric bilinear map 
$\vre:\cd(\vre)\times\cd(\vre)\longrightarrow\R$
is called {\bf closed form} on  $L^{2}(X,m)$ provided it is positive
(i.e. $\vre(u,u)\geq 0,$ for each $u\in\cd (\vre )$)
and $\cd(\vre )$ endowed with the scalar  product
$\vre_{1}(u,v):=\vre(u,v)+(u,v)_{2}, \quad u,v\in\cd(\vre)$
is a Hilbert   space. 
We have denoted by  $(\, ,\,)_{2}$ the scalar product from
$L^{2}(X, m):$ 
$(u,v)_{2}:= \int_{X}uvdm$.

A {\bf normal contraction} is a  function
$T:\R\rightarrow\R$ such that $T(0)=0$ and
$
\mid T(x)-T(y)\mid\leq\mid x-y\mid\quad \mbox{ for all }\, x,y\in\R.
$
An  example of normal contraction is the {\bf  unit contraction}
$T_{1}:\R\rightarrow\R$, defined by
$
T_{1}(x) := (x\vee 0)\wedge1.
$

A closed form  $\vre$ on  $L^{2}(X,m)$, having the domain 
$\cd(\vre)$ (we write
$(\vre,\cd(\vre)$), is called  {\bf Dirichlet  form} if
the unit contraction operates  on  $\vre$, that is:
$$
u\in\cd(\vre)\Longrightarrow(u\lor 0)\wedge 1 \in\cd(\vre)\mbox{ and }
\vre((u\lor0)\wedge1,(u\lor0)\land1)\le\vre(u,u), 
$$
where we have denoted by $\lor$, $\wedge$  the lattice  operations in 
$L^{2}(X, m)$.

Further we suppose that $X$ is a locally compact separable Hausdorff topological space,
$\cb$ is the  $\sigma$-algebra  of Borel measurable subsets of $X$, and
$m$ is a positive Radon measure, 
having as support the whole space $X$.
We denote by  $C(X)$ (respective $C_{c}(X)$) 
the set of all continuous (respective continuous with compact support)
real valued functions on $X$.

Let  $(\vre,\cd(\vre))$ be a Dirichlet  form.
Then $\vre$ has the following properties:

\noindent
$1)$ Each normal contraction operates on  $\vre$. 

\noindent
$2)$ If $u,v\in\cd(\vre)$ then 
$u\wedge v,u\vee v,u\wedge 1\in\cd(\vre)$. 

\noindent
$3)$ If $u,v\in\cd(\vre)\cap L^{\infty}(X,m)$ then 
$u\cdot v\in\cd(\vre)$ \c and 
$\vre(u\cdot v)^{\frac{1}{2}}\le
   \| u \|_{\infty}\vre(v)^{\frac{1}{2}}$
+ $\| v \|_{\infty}\vre(u)^{\frac{1}{2}}$
where $\|\,\cdot \|_{\infty}$ is the norm in $L^{\infty}(X,m)$ and 
$\vre (u)= \vre (u,u)$. 

\noindent
$4)$ If  $u\in\cd(\vre)$ and $u_{n}:= ( (-n)\lor u)\land n$, $n\in\N$, 
then $u_{n}\in\cd(\vre)$ \c and $u_{n}\longrightarrow u$
(when $n\rightarrow\infty$)
in the norm $\vre_{1}^{\frac{1}{2}}$.

Recall that the Dirichlet form  $(\vre,\cd(\vre))$  on  $L^{2}(X,m)$ is termed
{\bf regular} if the set
$\cd(\vre)\cap C_{c}(X)$ is  
dense in  $C_{c}(X)$ in the uniform norm and 
dense in  $\cd(\vre)$  in the norm  $\vre_{1}^{\frac{1}{2}}$
(induced by the scalar product $\vre_{1}$ of the Hilbert space $\cd(\vre)$);
see e.g. \cite{Fu}.

{\it In the sequel $(\ce,\cd(\ce))$ will be a regular
Dirichlet form on $L^{2}(X,m)$.}

\noindent
{\bf The capacity, quasi-continuity.}  
For each open set $G\subset X$ define
$\mbox{\sf cap}(G):=\inf 
\{ \ce_{1}(u,u)/ \; u\in\cd(\ce),\, u\geq 1\quad m \mbox{-a.e.  on  } G \}$
with the convention $\mbox{\sf cap}(G)=\infty$, 
whenever there is no $u\in\cd(\ce)$  with $u\geq 1$ $m$-a.e. on $G$.
For an arbitrary set $A\subset X$ we put
$\mbox{\sf cap}(A)=
\inf\{\mbox{\sf cap}(G)/\; G \mbox{ open},\, A\subset G\}.
$
In this way we obtained a {\bf Choquet capacity} on $X$.
A function $u:X\longrightarrow \R$ is called 
$\ce$-{\bf quasi-continuous}
if for each $\delta> 0$  there exists an open set  $G$  such that
$\mbox{\sf cap}(G)< \delta$ and $u|_{X\setminus G}$ is continuous.
Let $A\subset X$.
A property depending on $x\in A$  holds  
{\bf quasi everywhere on} (q.e. on) $A$
if there exists a set $N$ 
such that $\mbox{\sf cap}(N)=0$ and the property is true for
each $x\in A\setminus N$.
Since $(\ce,\cd(\ce))$ is regular then: 
{\it for each  $u\in\cd(\ce)$ there exist a  $m$-version $\tilde u$ of $u$
(i.e. $u=\tilde u$ $m$-a.e. on $X$) 
that is $\ce$-quasi-continuous}. 
Each other $\ce$-quasi-continuous $m$-version of $u$ coincides  q.e. 
with $\tilde u$.   
(see e.g. Theorem 3.1.3 \c and Lemma 3.1.4 in \cite{Fu}).

\noindent
{\bf The energy measure}. We present some basic facts about
the energy measure 
following the works \cite{LJ}, \cite{BM2}, \cite{Mo} and especially 
 \cite{FOT}, \cite{BH2} and \cite{St1}.
For each $f,u\in\cd(\ce)\bigcap L^{\infty}(X,m)$ 
the following inequality holds  (see $(3.2.13)$ from  \cite{FOT}):
$$
2\ce(u\cdot f,u)-\ce(u^{2},f)
\leq 2\|f\|_{\infty}\cdot\ce(u,u).
$$
In addition,  if 
$
f\geq 0$ then 
$0
\leq 2\ce(u\cdot f,u)-\ce(u^{2},f).$
Consequently, for each 
$u\in\cd(\ce)\cap L^{\infty}(X,m)$ 
there exists a uniquely determined positive Radon measure
$\mu_{\langle u \rangle }$ on  $X$ such that
$$
\int_{X}fd\mu_{\langle u \rangle }=2\ce(u\cdot f,u)-\ce(u^{2},f)\quad
\mbox{ for all } \, f\in\cd(\ce)\cap C_{c}(X).
$$
We have 
$\mu_{\langle u \rangle }(X)\leq 2\ce(u,u)<\infty$
and therefore $\mu_{\langle u \rangle }$ is a finite measure.
The measure  $\mu_{\langle u \rangle }$ is called  {\bf the energy measure} of
$u\in\cd(\ce)\cap L^{\infty}(X,m)$.
\noindent
If  $u,v\in\cd(\ce)\cap L^{\infty}(X,m)0$ then we define 
$$
\mu_{\langle u,v\rangle }:=\frac{1}{2}( \mu_{\langle u+v\rangle }-\mu_{\langle u\rangle }-\mu_{\langle v\rangle} ).
$$
Note  that $\mu_{\langle u,v\rangle }$ is the unique signed measure on $X$ such that 
$$
\int_{X}fd\mu_{\langle u,v\rangle }=
\ce(uf,v)+\ce(vf,u)-\ce(uv,f)\quad
\mbox{ for all }  \, f\in\cd(\ce)\cap C_{c}(X).
$$
Let $u\in\cd(\ce)$ \c si $(u_{n})_{n}\subset\cd(\ce)\cap L^{\infty}(X,m)$ 
a sequence converging to  $u$ in  the  norm $\ce_{1}^{\frac{1}{2}}$.
We define the {energy measure} $\mu_{\langle u \rangle}$ {of} $u$ by
$\mu_{\langle u\rangle }(f):=\lim_{n\rightarrow\infty}\mu_{\langle u_n \rangle } (f)$ for all 
$f\in C_{c}(X).$
 One can see that the positive Radon measure
$\mu_{\langle u \rangle }$ is well defined and
$\mu_{\langle u \rangle }(X)\leq 2\ce(u,u).$ 
For each $u\in\cd(\vre)$, the energy  measure $\mu_{\langle u\rangle }$ charges no set
of capacity zero (cf. Lemma 3.2.4 din \cite{FOT}).

{
(The Cauchy-Schwarz inequality; cf. \cite{St1}.)
Let $u,v\in\cd(\ce)$ \c  and $f , g$
be two  bounded $\cb$-measurable functions.
Then
$
\int_{X} f\cdot gd\mu_{\langle u,v\rangle }\leq\left(
\int_{X} f^{2}d\mu_{\langle u \rangle}\right)^
{\frac{1}{2}}
\cdot\left(
\int_{X} g^{2}d\mu_{\langle v \rangle }\right)^{\frac{1}{2}}\leq
$
$
\frac{1}{2}
\left(\int_{X}f^{2}d\mu_{\langle u\rangle }+
\int_{X} g^{2}d\mu_{\langle v \rangle }\right).
$}

The following assertions are equivalent: 

$i)$ The form $(\ce,\cd(\ce))$ is strongly local. 

$ii)$ $1_{G}\cdot d\mu_{\langle u \rangle }=0$ for each $u\in\cd$ \c  and 
relatively compact open set $G$, $u$ constant $m$-a.e. on $G$. 

$iii)$ For each $u,v\in\cd\cap L^{\infty}(X,m)$ \c and $w\in\cd$ we have
 $\mu_{\langle u\cdot v,w \rangle }=\tilde u\cdot\mu_{\langle v,w \rangle }+\tilde v\cdot\mu_{\langle u,w\rangle }$ (the Leibniz rule).

Assume that $(\ce,\cd(\ce))$ is strongly local and
let $\varphi\in C^{1}(\R^{n})$ and $u_{1},...,u_{n}\in\cd_{b,loc}$ , 
$\underline{u}:=(u_{1},...,u_{n})$.
Then the following equality (called {\it  the chain rule}) holds
(cf. Theorem 3.2.2 in \cite{FOT}):
$
\mu_{\langle \varphi(\underline{u}),v \rangle }=
\sum_{i=1}^{n}\frac{\partial\varphi}{\partial x_{i}}(
\underline{\tilde u})\cdot\mu_{ \langle u_{i},v \rangle }\quad \mbox{ for all }\, v\in\cd_{b,loc}.
$
If in addition the partial derivatives 
$\frac{\partial\varphi}{\partial x_{i}}$ 
are uniformly bounded, then above equality holds for all
$u_{1},...,u_{n},v\in\cd_{loc}$.
If  $u,v,w\in\cd_{loc}$ then the following {\bf truncation formula} holds  (cf. \cite{BM2}, \cite{Mo} and \cite{St1}):

\vspace{-4mm}

\begin{equation} \label{6.1}
\mu_{\langle u\land v,w \rangle }=1_{[\tilde u<\tilde v]}\cdot
\mu_{\langle u,w \rangle }+1_{[\tilde u\geq\tilde v]}\cdot\mu_{\langle v,w \rangle },\quad
\mu_{\langle u\land v,u\land v \rangle }=
1_{[\tilde u<\tilde v]}\cdot\mu_{\langle u,u\rangle}+
1_{[\tilde u\geq\tilde v]}\cdot\mu_{\langle v,v \rangle}.
\end{equation}
In addition, 
\begin{equation} \label{7.2}
\mbox{if } 
A\in\cb  \mbox{ and }
u=v  \mbox{  q.e. on  } A \mbox{  then }
1_{A}\cdot\mu_{\langle u,w\rangle }=1_{A}\cdot\mu_{\langle v,w \rangle }.
\end{equation}

A Dirichlet form $(\vre,\cd(\vre))$ on $L^{2}(X,m)$ is called
{\bf local} if for each  $u,v\in\cd(\vre)$  with compact supports 
$\mbox{\sf supp}u$, $\mbox{\sf supp}v$ we have
$ \vre(u,v)=0$ provided that 
$\mbox{\sf supp}u\cap\mbox{\sf supp}v = \emptyset$,  
where  if $u\in L^{2}(X,m)$ we  denoted  by $\mbox{\sf supp} u$
the support of the measure $u^{2}\cdot m$.

The form $(\ce,\cd(\ce))$ is termed {\bf strongly local}  if for each 
$u,v\in\cd(\ce)$  with  $\pp u$, $\pp v$ compact and 
$v$ constant on a neighbourhood of  $\pp u$ we have  $\ce(u,v)=0$.

{\it Further we  suppose that  $(\ce,\cd(\ce))$ 
is a strongly local regular Dirichlet form.}

The following equality holds:
$$
\ce(u,v)=\frac{1}{2} \int_{X}d\mu_{<u,v>}
\quad \mbox{ for all } \, u,v\in\cd(\ce).
$$

Define the {\bf local  domain} of the form by 
$\cd_{loc}$:=
$\lbrace u:X\rightarrow\R: $ for each relatively compact open set
$G$ there exists $w\in\cd(\ce)$ with  $u=w,$  $m$-a.e. on  $G\rbrace .$
If  $u\in\cd_{loc}$ we define  the {energy measure}   
$\mu_{<u>}$  of  $u$  by
$$
\mu_{<u>}:=\mu_{<u_n>}\mbox{ on } G_{n}, 
$$
where  $(G_{n})_{n}$ is an exhaustion of the space with 
relatively compact open sets,
$\bar G_{n}\subset G_{n+1}$, for each $n$,
and $u_{n}\in\cd$, $u_{n}=u$ $m$-a.e. on  $G_{n}$.
One can see that  the 
$\sigma$-finite measure  $\mu_{<u>}$ is well defined.
Denote by $\cd_{b,loc}$ the subspace of  $\cd_{loc}$ 
of all locally bounded functions.

The form $(\ce,\cd(\ce))$ admits a {\bf carr\'{e} du champ operator}  provided that
there exists a subspace $H\subset\cd(\ce)\cap L^{\infty}(X,m)$
dense in  $\cd(\ce)$ such that for each $u\in H$ 
there exists  $\overline{u}\in L^{1}(X,m)$  with
$$
2\ce(uf,u)-\ce(u^{2},f)=
\int_{X} f\overline{u}dm
\quad \mbox{ for all } \, f\in\cd(\ce)\cap L^{\infty}(X,m); 
$$
cf. \cite{BH2}.
By Proposition 4.1.3 in  \cite{BH2} we deduce that:
\noindent
$(\ce,\cd(\ce))$ admits a carr\'{e} du champ 
if and only if
for each $u,v\in\cd(\ce)$, the energy measure
$\mu_{<u,v>}$ is absolutely continuous with respect to $m$.
Let $\Gamma(u,v)$ denote its appropriate  Radon-Nikodym  density,
$
\mu_{<u,v>}=\Gamma(u,v)\cdot m .
$
We define in this way a positive symmetric bilinear form
$$
\Gamma:\cd(\ce)\times\cd(\ce)\longrightarrow L^{1}(X,m)
$$
such that
$
\ce(u\cdot f,v)+\ce(v\cdot f,u)-\ce(u\cdot v,f)=
\int_{X} f\Gamma(u,v)dm\quad
\mbox{ for all } \, u,v,f\in\cd\cap L^{\infty}(X,m).
$
Whenever $(\ce,\cd(\ce))$ admits a carr\'{e} du champ operator, 
the  form $\Gamma$ is called the {\bf carr\'{e} du champ operator} 
associated with  $\ce$.

Assume that  $(\ce,\cd(\ce))$  is a strongly local 
Dirichlet form admitting a
carr\'{e} du champ operator $\Gamma$. 
We say that 
$(\ce,\cd(\ce))$ 
{\bf admits a gradient $D$} (cf.  Definition 5.2.1 in \cite{BH2})
provided that there exists a separable Hilbert space $H$, a vector subspace $\cd_0$ dense 
in $\cd(\ce)$ and a linear map $D$ from $\cd_0$ into $L^2(m; H)$
such that
\begin{equation}\label{gradient} 
\Gamma(u)= \| Du \|^2_H \ \mbox{ for all } u\in \cd_0.
\end{equation}

The next proposition collects the basic properties of the carr\'e du champ operator.

\begin{prop} \label{prop6.1} 
Let $(\ce,\cd(\ce))$  be a strongly local 
Dirichlet form admitting a
carr\'{e} du champ operator $\Gamma$. .
Then the carr\'{e} du champ  operator  $\Gamma$ 
has the following properties.

1) We have
$$ 
\ce(u,v)=\frac{1}{2}
\int_{X}\Gamma(u,v)dm\quad
\mbox{ for all }\, u,v\in\cd(\ce); 
$$ 
see Proposition  6.1.1 in \cite{BH2}.

2) If ll  $\varphi\in C^{1}(\R^{n})$ with $\varphi(0)=0$  then
$$ 
\Gamma(\varphi(\underline{u}), v)=
\sum_{i=1}^{n}\frac{\partial\varphi}{\partial x_{i}}(
\underline{\tilde u})\Gamma(u_{i},v)\quad \mbox{ for all }\, u_{1},...,u_{n},v\in
\cd\cap L^{\infty}(X,m); 
$$ 
cf. Corollary 6.1.3 in \cite{BH2}.

3) For all $ \, u,v, w \in\cd(\ce)$ we have
\begin{equation} \label{6.2}
\Gamma(u\land v,w)=1_{[\tilde u<\tilde v]}\cdot
\Gamma(u,w)+1_{[\tilde u\geq\tilde v]}\cdot\Gamma(v,w), 
\end{equation}
\begin{equation} \label{6.3}
\vert\Gamma(u,v)\vert\leq
\Gamma(u)^{\frac{1}{2}} \Gamma(v)^{\frac{1}{2}},
\end{equation}
\begin{equation} \label{6.4}
\Gamma(u\wedge v)^p + \Gamma(u\vee v)^p = \Gamma(u)^p + \Gamma(v)^p ,
\end{equation}
where $\Gamma(u):=\Gamma(u,u)$, $p>0$, and
\begin{equation} \label{6.5}
\Gamma(u+v)
^{\frac{1}{2}}\leq\Gamma(u)^{\frac{1}{2}}+\Gamma(v)^{\frac{1}{2}}.
\end{equation}

4) 
If $A\subset X$, $A\in\cb$ \c and $u=v$ q.e. on $A$, then
$\Gamma(u,w)=\Gamma(v,w)\ m \mbox{-a.e. on }A \ 
\mbox{ for all } \, w\in\cd(\ce) .$
Particularly, if  $u\in\cd(\ce)$ is constant  $m$-a.e. on an open set 
$G$  then
$
\Gamma(u)=0\quad
m \mbox{-a.e. on }G.
$

5) 
 The Dirichlet form $(\ce,\cd(\ce))$ 
{admits a gradient $D$}. 
The operator $(\cd_0, D)$ is closable as an operator from $L^2(m)$ in $L^2 (m; H)$, its closure still denoted by $D$ has $\cd(\ce)$ as its domain and it is a continuous operator from 
$\cd(\ce)$ into $L^2 (m; H)$.
\end{prop}

Further we present several arguments for the proof of Proposition \ref{prop6.1}.

\noindent
{\it Proof of $(\ref{6.4})$}. By  $(\ref{6.2})$ we have 
$\Gamma(u\land v)=1_{[\tilde u<\tilde v]}\cdot
\Gamma(u)+1_{[\tilde u\geq\tilde v]}\cdot\Gamma(v)$. 
We also have 
$\Gamma(u\vee v)= 
\Gamma((-u)\wedge (-v))= 1_{[\tilde u<\tilde v]}\cdot
\Gamma(v)+1_{[\tilde u\geq\tilde v]}\cdot\Gamma(u)$. 
Therefore
$\Gamma(u\land v)^p=1_{[\tilde u<\tilde v]}\cdot
\Gamma(u)^p+1_{[\tilde u\geq\tilde v]}\cdot\Gamma(v)^p$
and 
$\Gamma(u\vee v)^p=1_{[\tilde u<\tilde v]}\cdot
\Gamma(v)^p+1_{[\tilde u\geq\tilde v]}\cdot\Gamma(u)^p$.
The relation $(\ref{6.4})$ follows now by adding the last two equalities.\\[1mm]
\noindent
{\it Proof of assertion $5)$  of Proposition \ref{prop6.1}.}  
Note first that the domain $\cd(\ce)$ of  the form is separable in the  norm $\ce_1^{\frac 1 2}$.
Indeed, let $\ca$ be a countable dense subset pf $L^2(X,m)$, fix $\alpha >0$ and let $U_\alpha$ be the $\alpha$-level operator of the resolvent family of the form 
$(\ce,\cd(\ce))$.
Then $\cd (L)= U_\alpha (L^2(X, m)))$, one can see that $U_\alpha (\ca)$ is dense in $\cd(L)$ in the graph norm,  
and since $\cd(L)$ is dense in $\cd(\ce)$ in the norm $\ce_1^{\frac 1 2}$  
we conclude that $U_\alpha (\ca)$ is also dense in $\cd(\ce)$ in the norm $\ce_1^{\frac 1 2}$,  as claimed.
Assertion $5)$ follows now by Proposition 5.2.2 $a)$ combined with Exercise 5.9 
(which is a result of G. Mokobodzki) from \cite{BH2}; alternatively,  see Theorem 3.9 from \cite{E}.\\

\noindent
{\bf A2. Proof of assertion $1)$ of Proposition \ref{prop3.2}.}
Let  $(u_n)_n\subset\cd_p$ be a Cauchy  sequence in the norm  $\|\cdot \|_{\cd_p}$.
It follows that  $(u_n)_n$ is a   Cauchy sequence  in 
$L^p(X,m)$ and $L^2(X,m)$.
In addition we have
\begin{equation}\label{6.6}
\lim_{n,m\rightarrow\infty}\int_{X}\Gamma(u_n-u_m)
^{\frac{p}{2}}dm=0
\end{equation}
and so
$
\lim_{n,m\rightarrow\infty}\int_{X}\Gamma(u_n-u_m)dm= 0. 
$
Recall that from assertion $1)$ of Proposition \ref{6.1} we have
$\ce(u)=\frac{1}{2}\int_{X}\Gamma(u)dm
$
and therefore
$\lim_{n,m\rightarrow\infty}\ce(u_n-u_m)=0.$
We get that $(u_n)_n$ is a  Cauchy sequence   in the Banach space
$(\cd(\ce),\ce_1^\frac{1}{2})$ and let $u\in\cd(\ce)$ be such that
$u_n\longrightarrow u$ in $\cd(\ce)$.
We  show  that $u_n\longrightarrow u$ 
also in the  norm $\| \cdot \|_{\cd_p}$,
that is $\lim_{n\rightarrow\infty}\| u_n-u\|_p=0$ and
$\lim_{n\rightarrow \infty} \|\Gamma(u_n-u)^\frac{1}{2}\|_p=0$.
The sequence $(u_n)_n$ being  Cauchy in the norm $\|\cdot \|_p$ and
convergent
in $\|\cdot \|_2$ to $u$, we deduce that $u_n\longrightarrow u$ in
$L^p(X,m)$.
From $(\ref{6.5})$ and $(\ref{6.6})$ we  obtain
$\lim_{n,m\rightarrow\infty}\|\Gamma(u_n-u)^\frac{1}{2}-
\Gamma(u_m-u)^\frac{1}{2}\|_p\leq\lim_{n,m\rightarrow\infty}\|
\Gamma(u_n-u_m)^\frac{1}{2}\|_p=0
$
and therefore the sequence
$(\Gamma(u_n-u)^\frac{1}{2})_n$ is Cauchy in $L^p(X,m)$.
In addition, since  $u_n\longrightarrow u$ in $\ce_1^\frac{1}{2}$
we derive that
$\Gamma(u_n-u)^\frac{1}{2}\longrightarrow 0$
in $L^2(X,m)$.
It results that  $\Gamma(u_n-u)^\frac{1}{2}\longrightarrow 0$ in 
$L^p(X,m)$
and we conclude that $(\cd_p,\|\cdot \|_{\cd_p})$ is a
Banach space. \hspace{\fill} $\square$

\noindent
{\bf A3. Proof of Lemma \ref{lem3.12}.}
Let $L$, $K$, and $F$ be compact subsets of $X$ with $L\subset K$. 
By the strong subadditivity and the monotonicity properties of ${\sf cap}_p$ (assertions $1)$ and $(2a)$ of Theorem \ref{thm3.11}) we get
${\sf cap}_p (K\cup F) + {\sf cap}_p (L)\leq  
{\sf cap}_p (K\cup(L\cup F))+ {\sf cap}_p K\cap(L\cup  F))\leq
{\sf cap}_p (K)+ {\sf cap}_p (L\cup F)$, 
hence 
${\sf cap}_p (K\cup F) - {\sf cap}_p (L\cup F) \leq
{\sf cap}_p (K) - {\sf cap}_p (L)$. 
We repeat this procedure for the compacts $E_i=K_i$ and $F_i = L_i$ and by induction we obtain
${\sf cap}_p (\bigcup_{i=1}^k K_i) - {\sf cap}_p (\bigcup_{i=1}^k L_i) =$$
{\sf cap}_p (\bigcup_{i=1}^{k-1} K_i\cup K_k) - {\sf cap}_p (\bigcup_{i=1}^{k-1} L_i\cup K_k) +
{\sf cap}_p (K_k \cup \bigcup_{i=1}^{k-1} L_i ) - {\sf cap}_p (L_k \cup \bigcup_{i=1}^{k-1} L_i ) \leq $$
\sum_{i=1}^{k-1} ( {\sf cap}_p (K_i) - {\sf cap}_p (L_i) ) +  {\sf cap}_p (K_k) - {\sf cap}_p (L_k)=
\sum_{i=1}^{k} ( {\sf cap}_p (K_i) - {\sf cap}_p (L_i) )$.
We conclude that $(\ref{3.8})$ holds for compact sets

We show now that $(\ref{3.8})$ also holds for open sets $E_i$ and $F_i$. 
Observe first that: 
if $K\subset \bigcup_{i=1}^k E_i$ and $L_i \subset F_i$ are compact sets with $ \bigcup_{i=1}^k L_i\subset K$,  
then the compact set
$K_i:= K\setminus \bigcup_{j=1, j\not=i} E_j$ is a subset of $E_i$ and it contains $L_i$.
Using this fact we deduce  that $(\ref{3.8})$ for open sets follows from $(\ref{3.8})$ for compact sets.
Analogously, $(\ref{3.8})$ for arbitrary sets $F_i \subset E_i$ follows from the case of open sets.
\hspace{\fill} $\square$

\vspace{1mm}

\noindent
{\bf A4. Quasiregular mappings.}
Let $\Omega$ be a domain n $\R^p$. 
A mapping $f:\Omega\longrightarrow\R^p$, 
$f=(f^1,...,f^p)$,  is called {\it quasiregular}
provided that the following conditions are satisfied:

$a)$ $f^i\in W^{1,p}_{loc}(\Omega) \,
\mbox{ for all } \, 1\leq i\leq p$; \

$b)$ There exists a constant  $K$, $1\leq K<\infty$ , such that 
\begin{equation} \label{7.9}
\| Df(x)\|^p\leq K J_f(x)\mbox{ for almost every } x\in\Omega, 
\end{equation}

\vspace{-3mm}

\noindent
where  $Df(x)$ is  identified with a linear mapping on $\R^p$,
$Df(x):\R^p \longrightarrow \R^p$, 
$
Df(x) e_i=\displaystyle\sum\limits_{j=1}^{p}
\frac{\partial f^j}{\partial x^i}(x) e_j , 
$
and the norm $\| Df(x)\|$ is understood  as the operator norm of the linear
map of the Euclidean space $\R^p$;
recall that $Df(x)$ is the {\it Jacobi matrix}  of $f$ which is meaningful at almost every point $x\in \Omega$ since $f$ is in the Sobolev
space  $W^{1,p}_{loc}(\Omega)$ and $J_f(x)$ is the {\it Jacobian},
$J_f(x):=\det Df(x).$

Recall that a mapping $f:\Omega\longrightarrow\R^p$ is called 
{\it quasiconformal } provided that it is quasiregular and homemorphism onto $f(\Omega)$.

Let $f$ be a quasiregular mapping.
The smallest  constant  $K\geq 1$ for which 
{$(\ref{7.9})$ holds is called the {\it outer dilatation} of $f$  in $\Omega$ and it is denoted by $K_0(f)$.}
The smallest  constant  $K'\geq 1$ for which 
$$
J_f(x)\leq K 'l (Df(x))^p \mbox{ for almost every } x\in\Omega
$$
is called {\it inner dilatation} of $f$
and it is denoted by $K_I (f)$.
Here
$l(Df(x)): =\displaystyle\min\limits_{\vert h\vert=1}\vert Df(x)h\vert $
and since  $f$ is quasiregular it follows that 
$K_I(f)<\infty$.

We consider the matrix $\theta_{f}$  associated with $f$,
\begin{equation} \label{7.10}
\theta_f(x) := J_f(x)^\frac{2}{p}[Df(x)]^{-1}[D^{*}f(x)]^{-1},
\end{equation}
where 
$D^*f(x)=(Df(x))^*$ is the  transpose of the
Jacobi matrix $Df(x)$.
Since $f$ is quasiregular, the following uniform ellipticity condition
holds  for all  $\xi\in\R^p$:
\begin{equation} \label{7.11}
K_0(f)^{-\frac{2}{p}}\vert\xi\vert^2 \leq
( \theta_f(x)\xi, \xi ) \leq
K_I(f)^\frac{2}{p}\vert\xi\vert^2;
\end{equation}
see \cite{Re1}  or \cite{Re2}.
We get that the $G =\theta_{f}$ is a symmetric, positive definite   
$p\times p$-matrix of Borel functions
on  $\Omega$, satisfying the uniform ellipticity condition.\\[2mm]



\noindent
{\bf Funding } 
We gratefully acknowledge financial support by the Deutsche Forschungsgemeinschaft
(DFG, German Research Foundation) -- Project-ID 317210226 -- SFB 1283. 
This work was supported by a grant of the Ministry of Research, Innovation and Digitization, CNCS - UEFISCDI,
project number PN-III-P4-PCE-2021-0921, within PNCDI III. \\


\end{document}